\documentclass{article}
\usepackage{graphicx}
\usepackage[left=3cm, top=2cm, right=3cm, bottom=3cm]{geometry}
\usepackage{listings}
\usepackage{float}
\usepackage{amssymb,amsmath,mathtools}
\usepackage{enumerate}
\usepackage[numbers,square]{natbib}
\usepackage{wrapfig}
\usepackage{framed}
\usepackage{tabstackengine}

\usepackage{indentfirst}
\usepackage{setspace}
\usepackage{braket}
\usepackage[nottoc]{tocbibind}\settocbibname{References}
\usepackage{mathrsfs}
\usepackage{faktor}

\usepackage{tikz}
\tikzset{every loop/.style={}}
\usetikzlibrary{arrows,shapes,decorations.markings,automata,backgrounds,petri,bending,calc}

\usepackage{tikz-cd} 
\tikzset{
    labl/.style={anchor=south, rotate=90, inner sep=.5mm}
}

\usepackage[utf8]{inputenc}
\usepackage[english]{babel} 

\usepackage{amsthm}
\newtheorem{thm}{Theorem}[section]

\newtheorem{lemma}[thm]{Lemma}
\newtheorem{prop}[thm]{Proposition}
\newtheorem{cor}[thm]{Corollary}
\numberwithin{equation}{section}

\def\e{\varepsilon}
\def\ssm{\smallsetminus}
\def\W{\mathcal{W}}
\def\Z{\mathbb{Z}}
\def\<{\langle}
\def\>{\rangle}

\def\lkblu{{\rm{lk}_d}}

\theoremstyle{definition}
\newtheorem{defn}[thm]{Definition} 
\newtheorem{definition}[thm]{Definition} 
\newtheorem{example}[thm]{Example} 
\newtheorem{remk}[thm]{Remark}
\newtheorem{remark}[thm]{Remark}

\newenvironment{claim}[1]{\par\noindent\underline{Claim:}\space#1}{}

\DeclareMathOperator{\Aut}{Aut}
\DeclareMathOperator{\Map}{Map}

\newcommand{\bfP}{\mathbf{P}}
\newcommand{\bfF}{\mathbf{F}}
\newcommand{\vare}{\varepsilon}

\def\G{\Gamma}
\def\A{\mathcal A} 
\def\-{\overline}
\def\wh{\widehat}

\makeatletter
\newcommand{\subjclass}[2][2010]{%
	\let\@oldtitle\@title%
	\gdef\@title{\@oldtitle\footnotetext{#1 \emph{Mathematics subject classification.} #2}}%
}
\newcommand{\keywords}[1]{%
	\let\@@oldtitle\@title%
	\gdef\@title{\@@oldtitle\footnotetext{\emph{Key words and phrases.} #1.}}%
}
\makeatother

\title{Leighton's Theorem: extensions, limitations, and quasitrees}

\author{Martin R. Bridson and Sam Shepherd
	\thanks{Shepherd was funded by the Engineering and Physical Sciences Research Council.}} 


\date{\today}

\begin{document}
	
\maketitle

\begin{abstract}   
Leighton's Theorem states that if there is a tree $T$ that covers two finite graphs $G_1$ and $G_2$,
then there is a finite graph $\hat G$ that is covered by $T$ and covers both $G_1$ and $G_2$. We prove that this result
does not extend to regular covers by graphs other than trees. Nor does it extend to non-regular covers
by a quasitree, even if the automorphism group of the quasitree contains a uniform lattice. But it does extend to regular coverings by quasitrees. 
\end{abstract}

\bigskip 
\section{Introduction}

Leighton's Theorem \cite{leighton} states that if a pair of finite graphs $G_1, G_2$ have a common
covering space, then they have a common finite covering. This serves as a prototype for a type of theorem that one might
hope to prove in many different categories.
 In particular, for different categories of (orbi)spaces one might hope to prove
that if two compact spaces have a common covering of a specified kind, then there is a common intermediate covering
that is compact. The search for such theorems is related to the study of rigidity, particularly quasi-isometric rigidity:
if one is in a situation where one can promote the existence of a quasi-isometry of groups $\G_1\sim \G_2$ into the existence
of proper cocompact actions on a common space $X$, then proving that $\G_1$ and $\G_2$ are abstractly commensurable 
is equivalent to proving that the orbispaces $X/\G_1$ and $X/\G_2$ have a common finite covering. 
Beyond quasi-isometric rigidity, the study of Leighton-type theorems fits into a rich and active area of geometry and group theory:
in geometry, particularly low-dimensional topology, understanding the existence and nature of finite-sheeted coverings has been a 
dominant theme in recent years (e.g. \cite{HW}, \cite{agol}); in parallel, there have been significant advances in group theory concerning the difficulty of recognising finite-index subgroups and the extent to which groups of geometric interest
are determined by their lattice of finite-index subgroups
(e.g. \cite{BWilt}, \cite{B4-annals}).

Against this background, the second author and others have extended Leighton's Theorem in various ways \cite{Huang}, \cite{neumann}, \cite{TwogenLeighton}, \cite{samThesis},  \cite{graphfins} with applications to rigidity 
\cite{BehrstockNeumann}, \cite{sam-woodh}, \cite{StarkWoodhouse}.
Our main purpose in this article is to understand the extent to which Leighton's Theorem applies to coverings of graphs by
quasitrees, i.e. locally finite graphs that are quasi-isometric to trees.
We shall see that this is a subtle question, the answer to which depends crucially on whether the coverings
considered are regular or not and whether the automorphism group of the quasitree in question contains uniform lattices.   

Moving beyond quasitrees and virtually free groups, 
we shall also describe situations where finite graphs with a common regular covering do not share a common
intermediate covering because one of the Galois groups involved does not have sufficiently many (or indeed any) finite-index
subgroups (Theorem \ref{t:BMW}).

In the classical setting, it is easy to reduce Leighton's Theorem 
to the case where the graphs $G_1$ and $G_2$ are connected. Standard arguments about covering spaces show that in 
this setting the conclusion of the theorem is that if the universal 
coverings of $G_1$ and $G_2$ are isomorphic to the same tree $T$, 
then after composing with automorphisms of $T$, the covering maps $T\to G_i$ both factor through a covering $T\to\hat{G}$
where $\hat{G}$ is a finite graph. In particular, there is a diagram of covering maps
\begin{equation}\label{e:Leigt}
	\begin{tikzcd}[
	ar symbol/.style = {draw=none,"#1" description,sloped},
	isomorphic/.style = {ar symbol={\cong}},
	equals/.style = {ar symbol={=}},
	subset/.style = {ar symbol={\subset}}
	]
	&T\ar{d}\\
	&\hat{G}\ar{dl}\ar{dr}\\
	G_1&&G_2	
	\end{tikzcd}
\end{equation}
Moreover, one can arrange for all of these coverings to be regular: this is a useful refinement
made explicit in Bass and Kuhlkarni's proof \cite{BK}. If one thinks in terms of actions rather than
quotient spaces, then Leighton's Theorem is the torsion-free case of the following theorem:
if $\G_1,\G_2<{\rm{Aut}}(T)$ act properly and cocompactly on $T$, then
there exists $g\in{\rm{Aut}}(T)$ such that $\Gamma_1^g:=g^{-1}\G_1g$ and $\G_2$  are commensurable, i.e. $\Gamma_1^g\cap\G_2$
has finite index in both $\Gamma_1^g$ and $\G_2$. We shall rely on a strengthening of this formulation
proved in \cite{TwogenLeighton} that constrains the conjugator $g$.

Our first result shows that one cannot extend Leighton's Theorem to arbitrary
regular coverings $G\to G_i$ of finite connected graphs. 

\begin{thm}\label{t:BMW} There exist connected graphs $G, G_1, G_2$ and regular covering
maps $G\to G_1$ and $G\to G_2$, such that 
$G_1$ and $G_2$ are finite but there does not exist a covering $G\to\hat{G}$
of any finite graph $\hat{G}$ such that $\hat{G}$ covers both $G_1$ and $G_2$.
\end{thm}

Theorem \ref{t:BMW} illustrates the need to impose conditions on the common covering
of $G_1$ and $G_2$ as one moves away from trees. 
Our proof of this theorem encodes structures from another setting in which the natural analogue of Leighton's Theorem fails,
namely that of compact, non-positively curved squared complexes. 
Key examples in this setting were constructed by Wise \cite{wise} and Burger-Mozes
\cite{BM}.


Henceforth, we shall restrict our attention to the case where the common covering
of $G_1$ and $G_2$ is quasi-isometric to a tree, i.e.~is a {\em{quasitree}}. 
In the positive direction, relying on the Symmetry-Restricted Leighton Theorem from \cite{TwogenLeighton}, we shall prove the following result,
for which we need to establish some notation. For a locally finite graph $X$,  we write ${\rm{Aut}}(X)$
to denote the group of automorphisms of $X$ (i.e. the isometries that preserve the vertex set). We equip ${\rm{Aut}}(X)$
with the usual (compact-open) topology, making it a locally-compact Hausdorff group. 
A {\em{uniform lattice}} $\G < \rm{Aut}(X)$ is a discrete subgroup that acts properly and
cocompactly on $X$. 
For $H<\rm{Aut}(X)$, the set of automorphisms that coincide on balls of radius $R$ with elements of $H$ will be denoted by $\mathscr{S}_R(H,X)$; see definition \ref{def22}.

\begin{thm}\label{t:qL}
Let $X$ be a locally finite quasitree, let $H<\rm{Aut}(X)$, and let $\Gamma_1, \Gamma_2 < H$ be uniform lattices in ${\rm{Aut}}(X)$.
Then for all $R \in \mathbb{N}$ there exists $g \in \mathscr{S}_R(H,X)$ such that $\Gamma_1^g$ is commensurable to $\Gamma_2$ in ${\rm{Aut}}(X)$. 
\end{thm}

As a corollary of Theorem \ref{t:qL}, we deduce that any two finite simplicial complexes with free fundamental groups and a common universal cover must have a common finite cover (Corollary \ref{cor:simplicial}).

Let us concentrate now on the case where $H=\rm{Aut}(X)$ and $\Gamma_1$ and $\Gamma_2$ are torsion-free. In
this case Theorem \ref{t:qL} 
tells us that, given finite graphs $X_1$ and $X_2$,
if there is a quasitree $X$ and {\em{regular}} coverings $X\to X_i$, then there is a covering $X\to\hat{X}$ of a finite graph $\hat{X}$ such that $\hat{X}$ covers both $X_1$ and $X_2$. However, our next result shows that this does not remain
true if one drops the requirement that both of the coverings $X\to X_i$ are regular, even if one retains the condition
that $\rm{Aut}(X)$ contains a uniform
lattice (which implies that $\Aut(X)$ is unimodular). 
Note that, just like Theorem \ref{t:BMW}, this theorem excludes the existence of arbitrary coverings $X\to \hat{X}\to X_i$, not just those
for which the compositions are the original coverings $X\to X_i$.

\begin{thm}\label{t:unimodularexample}
There exists a quasitree $X$, finite graphs $X_1$ and $X_2$, and
covering maps
$X\to X_1$ and $X\to X_2$,  with $X\to X_2$ regular, 
for which there is no finite graph $\hat{X}$ fitting into the following diagram of covers.
	 \begin{equation}
	\begin{tikzcd}[
	ar symbol/.style = {draw=none,"#1" description,sloped},
	isomorphic/.style = {ar symbol={\cong}},
	equals/.style = {ar symbol={=}},
	subset/.style = {ar symbol={\subset}}
	]
	&X\ar{d}\\
	&\hat{X}\ar{dl}\ar{dr}\\
	X_1&&X_2	
	\end{tikzcd}
	\end{equation}	
\end{thm}

In Section \ref{s:cocompact} we shall prove a version of
Theorem \ref{t:unimodularexample} in which the regularity requirement on $X\to X_2$ is replaced by the weaker
requirement that $\rm{Aut}(X)$ acts cocompactly on $X$, which serves to fix the main ideas in the proof and expose the key phenomena. In 
Section \ref{s:unimodular} we develop the technical refinements needed to prove Theorem \ref{t:unimodularexample}. 

We have already alluded to the importance of the fact that when a tree covers a finite graph, the covering 
is regular.  
When one moves away from trees, regularity is a more subtle issue. We explore this theme in 
the final section of this paper, where we prove the following theorem, in which $\mathscr{C}(X,Y)$
denotes the compact space of local-isometries $X\to Y$ that are covering maps.

\begin{thm}\label{thm:probmeasure}
Let $X$ be a locally finite quasitree such that ${\rm{Aut}}(X)$ contains a uniform lattice.
Suppose that $X$ covers a finite graph $Y$. Then there exists an $\rm{Aut}(X)$-invariant probability measure on $\mathscr{C}(X,Y)$ if and only if there exists a finite cover $\hat{Y}\to Y$ that admits a regular covering map $X\to\hat{Y}$.
\end{thm}

The remainder of this paper is structured in accordance with the foregoing discussion. In Section \ref{s:no-cover}
we prove Theorem \ref{t:BMW}. In Section \ref{s:positive} we prove Theorem \ref{t:qL}; we do so by reducing to
the case where $X$ is a tree, which is the Symmetry-Restricted Leighton Theorem
proved by the second author and Gardam-Woodhouse in \cite{TwogenLeighton}. This reduction is based on the
existence of an $\rm{Aut}(X)$-equivariant quasi-isometry from $X$ to a locally finite tree (Theorem \ref{maketree}), which
is established using elements of the theory of CAT$(0)$ cube complexes. (An alternative proof can be based on 
\cite{MSW}; see Remark \ref{r:msw}). We then prove Theorem \ref{t:unimodularexample} in two steps over Sections \ref{s:cocompact} and \ref{s:unimodular}, as described above,
before closing with a proof of Theorem \ref{thm:probmeasure} in Section \ref{s:measures}.

We thank the referee for their careful reading and helpful comments.

\bigskip
\section{The need to control the geometry of the common cover} \label{s:no-cover}

We begin with a weak form of Theorem \ref{t:BMW} that exemplifies the main idea with very simple graphs $G_1, G_2$.
This result does not correspond to diagram (\ref{e:Leigt}) but rather to the stronger conclusion of Leighton's Theorem
that, given coverings $p_i:T\to G_i$ of finite graphs $G_i$ by a tree $T$, there is an automorphism $g$ of $T$ and
a commutative diagram of coverings with $\hat{G}$ finite
\begin{equation}\label{other-diagram}
\begin{tikzcd}[
ar symbol/.style = {draw=none,"#1" description,sloped},
isomorphic/.style = {ar symbol={\cong}},
equals/.style = {ar symbol={=}},
subset/.style = {ar symbol={\subset}}
]
T\ar{dd}[swap]{p_1}\ar{rr}{g}\ar{dr}[swap]{\tau_1}&&T\ar{dl}{\tau_2}\ar{dd}{p_2}\\
&\hat{G}\ar{dl}{\lambda_1}\ar{dr}[swap]{\lambda_2}\\
G_1&&G_2
\end{tikzcd}
\end{equation} 

\begin{prop}\label{prop} There exist connected graphs $G, G_1, G_2$ and covering
maps $p_i: G\to G_i\ (i=1,2)$ such that 
$G_1$ and $G_2$ are finite but there does {\bf{not}} exist a  finite graph $\hat{G}$
with covering maps $\lambda_i:\hat G\to G_i\ (i=1,2)$
and $\tau_i:G\to \hat{G}$
such that $p_i=\lambda_i\circ\tau_i\ (i=1,2)$.
\end{prop}

\begin{proof}
Let $\Pi=F_k/N$ be an infinite group that has no proper 
subgroups of finite index,
where $F_k$ is the free group on $\{a_1,\dots, a_k\}$. Let $G_1$ be the $k$-petal rose -- i.e. the graph with one vertex and
$E(G_1)=\{a_1,\dots, a_k\}$. Let $ G$ be the Cayley graph of $\Pi$ with respect to the generators $a_i$ and consider 
the regular covering $p_1:{G}\to G_1$ with Galois group $\Pi$. The intermediate covers ${G}\overset{\tau_1}\to \hat G\overset{\lambda_1}\to G_1$
correspond to the subgroups $H<F_k$ that contain $N$, and $\hat{G}\to G_1$ is a finite covering if and only if $H$ has finite index in $F_k$.
By assumption, there are no proper subgroups of finite index in $F_k/N$ and therefore
no proper subgroups of finite index in $F_k$ that contain $N$. Thus $\hat G=G_1$ if $\hat{G}$ is finite.

This general construction will provide us with examples as described in the 
proposition provided that we can exhibit pairs of  groups
$\Pi_1$ and $\Pi_2$ with the following properties:
\begin{enumerate}
\item there are  finite generating sets $\A_i$  for $\Pi_i$ (of the same cardinality $k$)
such that the corresponding  Cayley graphs $\mathcal{C}(\Pi_i,\A_i)$
(considered as unlabelled graphs) are isomorphic;
\item $\Pi_2$ has a proper subgroup of finite index, but $\Pi_1$ does not.
\end{enumerate}
With such a pair of groups in hand, we can complete the proof by arguing as follows.
Take $G_1$ to be the $k$-petalled rose with covering
$p_1:\mathcal{C}(\Pi_1,\A_1)\to G_1$ and take $G_2^\dagger$ to be the $k$-petalled rose
with covering $p_2^\dagger:\mathcal{C}(\Pi_2,\A_2)\to G_2^\dagger$. Then take $G_2$ to be a proper finite-sheeted
cover of $G_2^\dagger$ corresponding to a proper finite-index subgroup of $F_k$ that contains $\ker (F_k\to \Pi_2)$
and lift $p_2^\dagger$ to $p_2: \mathcal{C}(\Pi_2,\A_2)\to G_2$. $G_2$ is larger than $G_1$, so any $\hat{G}$ as in the proposition must be a proper cover of $G_1$, which would contradict the fact that $\Pi_1$ has no proper finite-index subgroups.

A rich source of examples $(\Pi_1,\Pi_2)$ with the required properties are the BMW
groups, which act freely on a product of two regular trees, with a single orbit of
vertices; the 1-skeleton of the product of trees will be the Cayley graph of $\Pi_i$ for a 
suitable choice of generators.
The direct product of two finitely generated free groups is such a group and this
will serve us as $\Pi_2$. But
there are many other fascinating examples -- see \cite{caprace} for a recent survey.
Of particular note are the examples constructed by Wise in his thesis \cite{wise} 
and the examples of Burger and Mozes
\cite{BM}, which are simple.
\end{proof}

\begin{example}\label{ex} 
The following group  was constructed by Wise \cite{wise}.  
The universal cover of the standard 2-complex of this presentation is the
product of a regular tree of valence 4 and a regular tree of valence 6. Thus
the Cayley graph  (as an unlabelled graph) is isomorphic to the standard
Cayley graph of $F_2\times F_3$.  But $\Pi$ does not contain a subgroup of finite index that splits as a
direct product of free groups. 
$$
\Pi=
\langle
a, b, x, y, z \mid
 aya^{-1}x^{-1}, byb^{-1}x^{-1}, azb^{-1}z^{-1}, axb^{-1}y^{-1}, bxa^{-1}z^{-1}, 
 bza^{-1}y^{-1}
 \rangle
$$
\end{example}

\subsection{Terminology for graphs, $\#$-subdivision and square-completion}

For the most part, we regard graphs as geometric objects. 
We write $V(X)$ for the vertex set of a graph $X$ and $E(X)$ for its set of (closed, unoriented) edges, and we
metrize $X$ as a geodesic metric space in which each edge has length $1$.
But it is also covenient to 
refer to the combinatorial structure of $X$; for example, we say that $X$ is {\em finite} if $V(X)$ and $E(X)$ are finite.
To work with edge-paths, it is convenient to replace  $E(X)$ by the set of oriented edges $E^{\pm}(X)$, which for each
$e\in E(X)$ contains the pair of local isometries $e^{\pm}:[0,1]\to X$ with image $e$ and $e^+(t)=e^-(1-t)$.
If $\e=e^\pm$ then $\-{\e}:=e^{\mp}$.
A {\em reduced circuit} in $X$ is the loop determined by a finite cyclically-ordered set $(\varepsilon_1,\dots,\varepsilon_n)$ with $\varepsilon_i\in E^{\pm}(X)$
and $\varepsilon_i(1)=\varepsilon_{i+1}(0),\ {\varepsilon_{i+1}}\neq\-{\varepsilon_i}$ for $i=1,\dots,n$, indices $\mod n$. (We specify {\em cyclic} ordering so that
$(\varepsilon_1,\dots,\varepsilon_n) = (\varepsilon_2,\dots,\varepsilon_n,\varepsilon_1)$ etc.) If the vertices $\varepsilon_i(1)$ are all distinct, then the subgraph consisting of these vertices and the unoriented edges corresponding to the $\varepsilon_i$ is called an \emph{$n$-cycle}.

\begin{definition} The {\em square completion} $\square (G)$ of a graph $G$ is the combinatorial 2-complex  
obtained from $G$ by attaching a 2-cell to each reduced circuit of length $4$.  
\end{definition}

\begin{definition}
Let $K$ be a squared 2-complex, i.e.~a combinatorial 2-complex such that the attaching map of each 2-cell is a 
reduced circuit of length $4$. We define $K_{\#}$ to be the 
squared 2-complex obtained from $K$ by introducing new vertices and edges so as to divide each edge of $K$ into a path of
combinatorial length $5$ and each 2-cell (square) into $25$ squares in the obvious $5$-by-$5$ pattern.
$K_{\#}$  is metrized as a piecewise-Euclidean complex where each (new) edge has length $1$ and each (new) $2$-cell is a square.
\end{definition}

The integer $5$ is used in this definition because it restricts the nature of short loops in $K_{\#}$;
the following lemma would fail if we used $2, 3$ or $4$.
We use the standard notation $K^{(1)}$ for the 1-skeleton of a combinatorial complex $K$.

\begin{lemma}\label{l:complete}
If a squared 2-complex $K$ is non-positively curved, then $K_{\#} = \square(K_{\#}^{(1)})$.
\end{lemma}

\begin{proof} For an arbitrary squared 2-complex, there are three types
of reduced circuits of length $4$ in $K_{\#}^{(1)}$, illustrated in Figure \ref{fig:three}. The first type consists
of the boundaries of the 2-cells in $K_{\#}$. The second and third types occur when the link of a vertex in $K$ contains a 1-cycle or 2-cycle respectively. If $K$ is non-positively curved then only circuits of the first type are possible.
\end{proof}

\begin{figure}[H]
	\scalebox{0.6}{
	\begin{tikzpicture}[
		node/.style={minimum size=.5cm-\pgflinewidth, outer sep=0pt}]
		\begin{scope}[shift={(2,0)}]
			\draw[step=1cm,color=black] (0,0) grid (5,5);
			\path[ultra thick] 
			(0,0) edge (5,0)
			(5,0) edge (5,5)
			(5,5) edge (0,5)
			(0,5) edge (0,0)
			(1,2) edge [blue] (1,3)
			(1,3) edge [blue] (2,3)
			(2,3) edge [blue] (2,2)
			(2,2) edge [blue] (1,2);
		\end{scope}
		
		\begin{scope}[shift={(10,0)}]
			\draw[step=1cm,color=black] (0,0) grid (5,5);
			\path[ultra thick] 
			(0,0) edge[postaction={decoration={markings,mark=at position 0.65 with {\arrow[black,line width=1mm]{triangle 60}}},decorate}] (5,0)
			(5,0) edge (5,5)
			(5,5) edge (0,5)
			(0,0) edge[postaction={decoration={markings,mark=at position 0.65 with {\arrow[black,line width=1mm]{triangle 60}}},decorate}] (0,5)
			(1,0) edge [blue] (1,1)
			(1,1) edge [blue] (1,2)
			(1,2) edge [blue] (0,2)
			(0,2) edge [blue] (0,1);
		\end{scope}
		
		\begin{scope}[shift={(18,0)}]
			\draw[step=1cm,color=black] (0,0) grid (10,5);
			\path[ultra thick] 
			(5,0) edge[postaction={decoration={markings,mark=at position 0.65 with {\arrow[black,line width=1mm]{triangle 60}}},decorate}] (0,0)
			(5,0) edge[postaction={decoration={markings,mark=at position 0.65 with {\arrow[black,line width=1mm]{triangle 60}}},decorate}] (10,0)
			(10,0) edge (10,5)
			(10,5) edge (0,5)
			(0,5) edge (0,0)
			(5,0) edge (5,5)
			(4,0) edge [blue] (4,1)
			(4,1) edge [blue] (5,1)
			(5,1) edge [blue] (6,1)
			(6,1) edge [blue] (6,0);
		\end{scope}
	\end{tikzpicture}
}
\caption{\small The three types of reduced circuits of length 4 in $K_{\#}^{(1)}$, highlighted in blue. The edges of $K$ are in bold, and identifications between them are indicated by arrows.}\label{fig:three}
\end{figure}
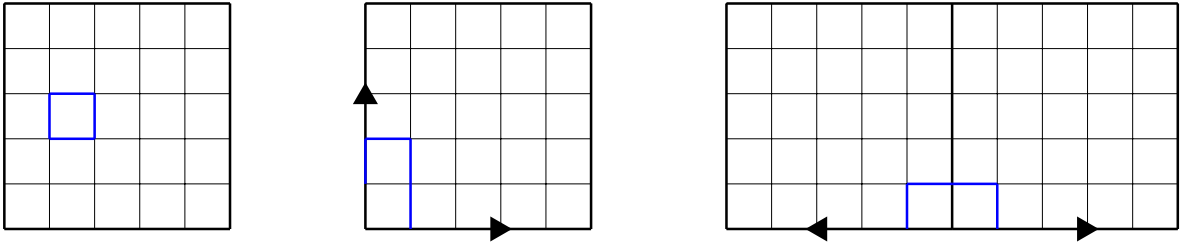

We stated the preceding lemma for non-positively curved complexes because that is our focus in this section, but the
proof shows that it is enough to assume that $K$ is {\em simple} in the sense that links of vertices 
in $K$ are simplicial graphs. 

\subsection{Proof of Theorem \ref{t:BMW}} 
In the proof of Proposition \ref{prop} we used the requirement that $p_i=\lambda_i\circ\tau_i$ to restrict our attention to intermediate covers of $p_i: G\to G_i$, but that is not good enough now  
because we do not require that
the leftmost and rightmost triangles in the analogue of diagram (\ref{other-diagram}) commute.
We shall remedy this by describing a 
more complicated graph $G_1$. 

To this end, we consider the quotient $K$ of a product of regular trees (the universal
cover $\tilde K$ of $K$)
by a BMW group; to be definite we take $K$ to be the standard 2-complex of
the presentation in Example \ref{ex}, which has fundamental group $\Pi$.
We subdivide $K$ to form $K_{\#}$, and subdivide its universal cover $\tilde K_{\#}$
in the same canonical manner. Lemma \ref{l:complete} 
assures us that $K_{\#}$ and $\tilde K_{\#}$ can be recovered from their 1-skeleta by taking the square-completion. 

The vertices of $G_1:=K_{\#}^{(1)}$ can be divided into three types: the original vertex of $K$ has valence $10$;
the new vertices introduced in the interior of 1-cells have valence $6$ or $8$; and the new vertices
introduced in the interior of 2-cells have valence $4$. A key point about the use of the integer
$5$ in the definition of $K_{\#}$ is that the only reduced circuits of length $4$ in $G_1$ are 
the boundary cycles of the 2-cells of $K_{\#}$. It follows that the number of such cycles passing through
each vertex is uniquely determined by the valence of that vertex, and this number is equal to
the number of $4$-cycles passing through a vertex of the same valence in $G:=\tilde K_{\#}^{(1)}$.

Let $\hat{G}$ be a graph and let  $\pi:\hat{G}\to G_1$ be a covering. As $G_1$ contains no reduced circuits of length 
less than $4$,
neither does $\hat{G}$. So the number of $4$-cycles passing through a vertex $v\in \hat{G}$ is no greater than the
number passing through $\pi(v)$. (It will be strictly less if some loop of length $4$ based at $\pi(v)$ does not
lift to a loop based at $v$.) Similarly, if $q:G\to \hat{G}$ is a covering map then the number of reduced
circuits of length $4$ based at $w\in V(G)$ will be no greater than the number passing through $q(w)$. 
In our setting, the number of such loops at $w$ is the same as the number at $\pi\circ q(w)$, because
these vertices have the same valence. Thus $q$ and $\pi$ both induce bijections on the set of reduced circuits
of length $4$ passing through each vertex.  
These circuits are the attaching maps for the 2-cells of the square-completion of each graph.
Therefore the coverings $G\to \hat{G}\to G_1$ extend to 
coverings of combinatorial 2-complexes $\tilde K_{\#}=\square(G)\to \square(\hat{G})\to \square(G_1) = K_{\#}$.
 
To complete the proof, it suffices to take $G_2$ to be 
the 1-skeleton of the quotient $Q$ of 
$\tilde K_{\#}$ by a direct product of free groups $F_2\times F_3$, acting freely and transitively
on the vertex set of $\tilde K$. Lemma \ref{l:complete} tells us that $Q= \square(G_2)$, so in particular
$\pi_1(\square(G_2))= F_2\times F_3$. 
Arguing as in the previous paragraph, we see that any coverings $G\to \hat{G}\to G_2$ will extend to 
coverings of combinatorial 2-complexes $\tilde K_{\#}=\square(G)\to \square(\hat{G})\to \square(G_2) = Q$.
Thus if there were a finite graph $\hat{G}$ that admitted coverings $G\to\hat{G}$ and $\hat{G}\to G_1,G_2$,
then $\pi_1(\square(\hat{G}))$ would be a subgroup of finite index in both $\Pi=\pi_1K_{\#}$ and $F_2\times F_3=\pi_1Q$.
Passing to a subgroup of finite index in $\pi_1(\square(\hat{G}))$  would then 
yield a direct product of free groups that had finite index in $\Pi$, which is a contradiction.  
\qed

 \bigskip
\section{Leighton's Theorem for coverings by quasitrees}\label{s:positive}

As we explained in the introduction, our strategy for proving Theorem \ref{t:qL} is to reduce it to the
case where $X$ is a tree by means of a general construction that promotes actions on quasitrees to actions on trees.
This is a variation on a result of Mosher, Sageev and Whyte \cite{MSW} (see Remark \ref{r:msw}).
In order to understand the proof that we shall present, the reader should be familiar with the basic properties
of CAT$(0)$ cube complexes \cite{BH}, in particular the definition and basic properties of hyperplanes and the duality
between cube complexes and spaces with walls. Key ideas in this theory originate in the work of  
Gromov \cite{gromov}, Sageev \cite{sageev}
and Haglund-Paulin \cite{HP}. There are several variants on the basic construction, adapted to different settings
(e.g.~Roller \cite{roller}, Nica \cite{nica}, Chatterji-Niblo \cite{chatt-niblo}). See
\cite{bestvina} and \cite{sageev-pcmi} for surveys. Wall spaces and group actions
on cube complexes play central roles in major recent advances in group theory and low-dimensional topology (
most dramatically \cite{agol}), following a programme of Wise \cite{riches}.

Our use of these ideas is typical in the subject. In brief, we identify a natural notion of a {\em wall}, which separates our underlying
space (the vertex set $V(X)$ of a quasitree) into two disjoint subspaces (half-spaces); 
the wall is said to {\em separate} $x,y\in V(X)$ if $x$ and $y$ lie in opposite halfspaces. It is required that
there be only finitely many walls separating each pair of points $x,y$. In our situation, 
every pair of distinct points is separated by at least one wall. An {\em ultrafilter} $\omega$
is a collection of half-spaces that gives a coherent choice of side across the collection of all walls:
if the pair of half-spaces $\{H, H'\}$ is a wall, then exactly one of $H,H'$ lies in $\omega$, and
if $\{H_1,H_1'\}$ and $\{H_2,H_2'\}$ are walls with $H_1\subset H_2$ and $H_1\in\omega$ then $H_2\in\omega$.
Each $x\in X$ defines an ultrafilter $\omega_x$ that picks out the half-spaces containing $x$, and the
vertex set of the {\em Sageev cube complex}
dual to the wall structure is the set of ultrafilters at finite distance from these $\omega_x$, where 
the distance between two ultrafilters is the number of walls for which the ultrafilters make a different choice of
half-space. The natural map $x\mapsto \omega_x$ embeds $V(X)$ in this vertex set.
The process of completing this vertex set to a cube complex  depends on the pattern of intersections
of the half-spaces associated to walls and is described in detail in each of the above references.

\begin{thm}\label{maketree}
There is a canonical process that, given a locally finite quasitree $X$ 
on which $G={\rm{Aut}}(X)$ acts cocompactly, 
will construct a locally finite tree $T$ with an action of $G$ on $T$
and a continuous $G$-equivariant quasi-isometry $f:X\to T$ that restricts
to an injection $V(X)\xhookrightarrow{}V(T)$.  
\end{thm}
\begin{proof}
	The idea of the proof is to define walls on $X$ that enable us embed it in a CAT$(0)$ cube complex $\Psi$, then use
	the panel-collapse procedure of Hagen and Touikan \cite{panel} to canonically retract $\Psi$ onto a tree.
	
	To this end, we define a constant $C_X$ by considering 
	all continuous quasi-isometries $h:X\to Y$ from $X$ to locally finite trees $Y$ that
	send vertices to vertices, and set 
	$$ 
	C_X:=3+\inf_{h:X\to Y}\sup_{e\in E(Y)}\text{diam}(h^{-1}(e))
	$$ 
	
For each set of edges $S\subset E(X)$, we let
$\hat{S}$ denote the set of midpoints, and we equip the vertex set $V(X)$ with the structure of a wall space by defining 
$\mathcal{W}$ to be the set of all $\hat{S}$ such that 
$X\ssm\hat{S}$ has exactly two connected components and the diameter of the union
of edges in $S$ is less than $C_X$. (More formally, the wall $\hat{S}$ is the partition of $V(X)$ into the intersections
of $V(X)$ with the two connected components of $X\ssm\hat{S}$.)
It is clear that any two points of $V(X)$ are separated by
only finitely many walls. Let $\Psi$ be the cube complex dual to this wall-space; its vertex set is defined in terms of
ultrafilters as above.
\\
	\begin{claim}
The canonical map $\theta:V(X)\to\Psi$ is injective.
	\end{claim}

To verify this basic property,  we must argue that
each pair of distinct vertices $x_1,x_2\in V(X)$ is separated by some $\hat{S}\in\mathcal{W}$. 
The set $S$ of edges incident at $x_1$ has diameter less than $C_X$ and separates $x_1$ from $x_2$.
If $S'\subseteq S$ is a minimal subset such that $\hat{S'}$  separates $x_1$ from $x_2$,
then $\hat{S'}\in\mathcal{W}$. This proves the claim.\\

The canonical
nature of the construction ensures that the action of $G={\rm{Aut}}(X)$  
on $V(X)$ extends to an action 
on $\Psi$ making $\theta$  equivariant. 
Each $\hat{S}\in\mathcal{W}$ can intersect only a uniformly
bounded number of other walls in $\mathcal{W}$, so the hyperplanes of $\Psi$ are finite and of bounded diameter.
The cocompactness
of the action of $G$ on $X$ implies that there are only finitely many $G$-orbits of walls,
so $G$ acts cocompactly on $\Psi$.\\
\begin{claim}
$\theta:V(X)\xhookrightarrow{}\Psi$ is a quasi-isometry with respect to the combinatorial metric on $\Psi$.
\end{claim}

As $G$ acts cocompactly on $\Psi$, every point is within a bounded distance of the image of $\theta$. 
By definition,
for $x_1,x_2\in V(X)$, the combinatorial 
distance between $\theta(x_1)$ and $\theta(x_2)$ is the number of walls $\hat{S}\in\mathcal{W}$ that separate $x_1$ and $x_2$, so we must find upper and lower bounds for this number of walls that are linear functions of $d(x_1,x_2)$.

To this end, we fix a tree $Y$ and  
a continuous quasi-isometry $h:X\to Y$ such that $2+\text{diam}(h^{-1}(e))<C_X$ for all $e\in E(Y)$ and
\begin{equation*}
\tfrac{1}{K}d(x_1,x_2)-K\leq d(h(x_1),h(x_2))\leq Kd(x_1,x_2)+K
\end{equation*}
for all $x_1,x_2\in V(X)$, where $K>1$ is constant.
For each $e\in E(Y)$, the union of the edges in
	$S(e):=\{e'\in E(X)\,|\,h(e')\cap e\neq\emptyset\}$ has
	diameter less than $C_X$ (if non-empty). Moreover, if $X\ssm\wh{S(e)}$ has more than one component, then $\wh{S'(e)}\in\W$ for some $S'(e)\subseteq S(e)$. 
There are at least $\tfrac{1}{K}d(x_1,x_2)-K-2$ edges 
on the geodesic in $Y$ joining $h(x_1)$ to $h(x_2)$, and for each such edge $\widehat{S(e)}$ will separate $x_1$ and $x_2$.
Since ${\rm{diam}}(h(e))\leq2K$ for all edges $e\in E(X)$, the sets  $S(e_1)$ and $S(e_2)$ are disjoint for $e_1,e_2\in E(Y)$ with $d(e_1,e_2)>2K$. Thus we have a
family of at least $(\tfrac{1}{K}d(x_1,x_2)-K-2)/(2K+2)$ disjoint sets $\wh{S'(e)}\in\mathcal{W}$ separating $x_1$ from $x_2$.

To obtain an upper bound on the number of walls separating $x_1$ and $x_2$, we fix
a shortest edge path $P$ from $x_1$ to $x_2$, and note that $\hat{S}$ can only separate $x_1$ and $x_2$ if $S$ includes an edge in $P$. Since $X$ is locally finite and cocompact, there is a uniform bound, $N$ say, on the number of sets $S\subset E(X)$ of diameter less than $C_X$ that contain any given edge $e\in E(X)$. Applying this to each edge in  $P$,
we see that at most $N\, d(x_1,x_2)$ walls separate $x_1$ from $x_2$.\\

\begin{claim}
$\Psi$ is locally finite.
\end{claim}

The hyperplanes in $\Psi$ are finite and each 
intersects only finitely many other hyperplanes, so if $\Psi$ were not locally finite, there would be
an infinite family of edges $e_i$ incident at some vertex $v\in\Psi$ with the
dual hyperplanes $H_i$ all disjoint. Let $\hat{S}_i\in\W$ be the wall corresponding to $H_i$ and fix $e'_i\in S_i$. The map $\theta$   sends the endpoints of $e'_i$ to different sides of the hyperplane $H_i$; let $x_i$ be the endpoint such that $H_i$
separates $\theta(x_i)$ from $v$. There are only finitely many $G$-orbits of pairs $(\hat{S},x)\in\mathcal{W}\times V(X)$ such that $x$ is an endpoint of an edge in $S$, so there is a uniform bound on the diameter of $H_i\cup\{\theta(x_i)\}$. And since the $H_i$
are all adjacent to $v$, this bounds the diameter of the set $\{\theta(x_i)\}$. 
But the $H_i$ are disjoint, so the $\theta(x_i)$ are all distinct vertices. This provides us with 
the contradiction that we seek, 
because the previous claim shows that if the diameter of the set $\{\theta(x_i)\}$ is bounded then so is
the diameter of  $\{x_i\}$, contradicting the local finiteness of $X$.\\

By passing to the first cubical subdivision, we can assume that $G$ acts on $\Psi$ without inversions in hyperplanes. Since $\Psi$ has finite hyperplanes and a cocompact $G$-action, we can apply the panel collapse procedure of Hagen and Touikan \cite{panel} 
 to obtain a locally finite tree $T$ with a $G$-equivariant embedding $T\xhookrightarrow{}\Psi$ that is a quasi-isometry and induces a bijection between the vertex sets. (The edges of $T$ will not in general map to edges of $\Psi$, they will just map to CAT$(0)$ geodesic segments between the appropriate vertices.) The fact that $T\xhookrightarrow{}\Psi$ is a quasi-isometry inducing a bijection between the vertex sets is not explicitly stated in the theorem of \cite{panel}, but it is obvious from the proof. By composing the inverse of this bijection with $\theta: V(X)\xhookrightarrow{} \Psi$ 
we get a $G$-equivariant quasi-isometry $f:V(X)\xhookrightarrow{} T$, and by extending linearly along edges we obtain a continuous $G$-equivariant quasi-isometry $f:X\to T$ that proves the theorem.
\end{proof}
\bigskip

\begin{remark}\label{r:not-minimal} 
The tree $T$ constructed above will not be a minimal $G$-tree in general, and the action of $G$ on the geodesic core of $T$
will not be faithful, even though the action of $G$ on $T$ is faithful.
For example, given any finite group $\Omega$, one can manufacture a quasi-line $X$ with
a cocompact action of the wreath product $G=\Omega\wr\Z$ by stringing 
together copies of the suspension of the Cayley graph of $\Omega$ in a linear fashion. The construction described in Theorem \ref{maketree}
will produce a $G$-equivariant map $X\to T$ where the geodesic core of $T$ is a simplicial line on which $\Omega$
will act trivially. The full tree $T$ is obtained from the core by attaching finite subtrees of uniformly
bounded diameter, and it is the action of
$G$ on these finite subtrees that makes $T$ a faithful $G$-tree.  In the case where $X$ is not a quasi-line, a similar
phenomenon still occurs with finite subgroups $\Omega<{\rm{Aut}}(X)$ that have finite support.
\end{remark}

\begin{remark}\label{r:msw}  Theorem \ref{maketree} is reminiscent of the work of Mosher, Sageev and Whyte \cite{MSW}
who were concerned with promoting quasi-actions on bushy quasitrees to genuine actions on trees. Indeed,
for quasitrees with infinitely many ends,  one
can craft a (less-canonical) alternative to Theorem \ref{maketree} by appealing to their work, as we now describe.

Recall that a tree is termed {\em bushy} if it has bounded valence and each vertex is a uniformly bounded distance from a vertex having at least three unbounded complementary components. 
Given a cocompact quasitree $X$ with infinitely many ends and
 $h:X\to Y$ a quasi-isometry to a tree, we modify $Y$ to make it bushy as follows. (The choice of $h:X\to Y$ is what makes this argument less canonical.) Because $X$ is cocompact,
 we can cover $X$ with translates of a finite subgraph $U\subset X$ that has at least three unbounded complementary components. 
 For suitable $R>0$, the $R$-neighbourhood of $h(g.U)$ has at least three unbounded complementary components in $Y$.
 It follows that every vertex of $Y$ is a uniformly bounded distance from one with at least three unbounded complementary components. Each such vertex lies in the core $Y'\subset Y$, which is the union of the geodesic lines. 
 Composing $h$ with the nearest-point retraction $Y\to Y'$, we obtain a quasi-isometry from $X$ to a bushy tree, and
 \cite[Theorem 1]{MSW} promotes the resulting
 quasi-action of $G={\rm{Aut}}(X)$ on $Y'$ to a cocompact isometric action on a locally finite tree $T$.
 This action is coarsely $G$-equivariant in the sense that there is a quasi-isometry $f:X\to T$ such that
  $d(f(g x),g f(x))$ is uniformly bounded independent of $g\in {\rm{Aut}}(X)$ and $x\in X$. 
 
Modifying $f$ by a bounded map, we can assume that it sends vertices to vertices. Then, to make it $G$-equivariant,
one decomposes $V(X)$ into $G$-orbits, and for each orbit representative $v_i$ one defines $f'(g.v_i)$ to be the
centre of the bounded set $gG_i.f(v_i)$, where $G_i<G$ is the stabiliser of $v_i$.
(It may be necessary to subdivide $T$ to ensure that this centre is a vertex.) Extend $f'$ linearly across edges to get a $G$-equivariant map $f':X\to T$.

The final thing to arrange is that $f'$ should be injective on $V(X)$. This can be done by adding various decorations. For
example, one might add a leaf labelled $e_x$ to $f'(x)$ for each $x\in V(X)$, and perturb $f'$ equivariantly so that it sends 
$x$ to the new vertex at the end of $e_x$.
\end{remark}

\subsection{Proof of Theorem \ref{t:qL}}

In what follows $B_R(x)$ will denote the $R$-ball centred on a vertex $x$ of a graph.
\begin{defn}\label{def22}
	Let $X$ be a graph, let $H<{\rm{Aut}}(X)$, and let $R$ be an integer. We define the \emph{$R$-symmetry restricted closure} of $H$ to be: 

$$	\mathscr{S}_R(H,X) := \{ g \in {\rm{Aut}}(X)\mid \forall x \in V(X), \exists h \in H \text{ s.t. $h$ agrees with $g$ on $B_R(x)$}\}
	$$
	It is easy to check that $\mathscr{S}_R(H,X) = \bigcap_{x \in V(X)} H\, {\rm{Fix}}(B_R(x))$, where 
	${\rm{Fix}}(B_R(x))$ is the pointwise stabiliser of $B_R(x)$ in ${\rm{Aut}}(X)$. If $X$ is locally finite then the topological group
	$\rm{Aut}(X)$ has a basis of open
	sets consisting of cosets of pointwise stabilisers of balls, and $\mathscr{S}_R(H,X)$ is a closed subgroup. 	
	A discrete subgroup $\Gamma<{\rm{Aut}}(X)$ is a \emph{uniform lattice} if it acts properly and cocompactly on $X$.
\end{defn} 

We recall the statement of Theorem \ref{t:qL} for the convenience of the reader. In the case where $X$ is a tree, 
this is the {\em Symmetry-Restricted Leighton Theorem} proved in \cite{TwogenLeighton}, and we rely on this basic case.
We do not need to assume that the groups acting are free because any group that acts properly and cocompactly on
a quasitree is virtually free.

\begin{thm}\label{t:qL2}
Let $X$ be a locally finite quasitree graph, let $H<\rm{Aut}(X)$, 
and let $\Gamma_1, \Gamma_2 < H$ be uniform lattices in ${\rm{Aut}}(X)$.
Then for all $R \in \mathbb{N}$ there exists $g \in \mathscr{S}_R(H,X)$ such that $\Gamma_1^g$ is commensurable to $\Gamma_2$ in ${\rm{Aut}}(X)$. 
\end{thm}

\begin{proof}
By subdividing we may assume that $X$ is simplicial. By Theorem \ref{maketree} there exists a locally finite tree $T$, an injective homomorphism $\iota:\rm{Aut}(X)\to\rm{Aut}(T)$ and a continuous $\iota$-equivariant quasi-isometry $f:X\to T$ which is injective on $V(X)$. As $\Gamma_1$ and $\Gamma_2$ act properly and cocompactly on $X$,
the groups  $\iota(\Gamma_1)$ and $\iota(\Gamma_2)$ act properly and cocompactly on $T$, i.e.
they are uniform lattices in $\rm{Aut}(T)$. Let $M$ be large enough so that for all $x\in V(X)$
\begin{equation}\label{Mdef}
f(B_R(x))\subset B_M(f(x)).
\end{equation}
We then have a homomorphism 
\begin{equation*}
\alpha:\mathscr{S}_M(\iota(H),T)\to\mathscr{S}_R(H,X),
\end{equation*}
satisfying $\alpha\circ\iota=\rm{id}$ on restriction to $\mathscr{S}_R(H,X)$, defined as follows. Take $\phi\in\mathscr{S}_M(\iota(H),T)$. For $x\in V(X)$ there exists $h\in H$ such that $\iota(h)$ agrees with $\phi$ on $B_M(f(x))$. By $\iota$-equivariance of $f$ and (\ref{Mdef}), we get the following commutative diagram.

\begin{equation}\label{xgcommute}
\begin{tikzcd}[
ar symbol/.style = {draw=none,"#1" description,sloped},
isomorphic/.style = {ar symbol={\cong}},
equals/.style = {ar symbol={=}},
subset/.style = {ar symbol={\subset}}
]
B_R(x)\arrow{r}{h}\arrow{d}{f}&X\arrow{d}{f}\\
B_M(f(x))\arrow{r}{\iota(h)}\arrow[bend right]{r}[swap]{\phi}&T
\end{tikzcd}
\end{equation}
We define the restriction of $\alpha(\phi)$ to $B_R(x)$ to agree with $h$. This gives a well-defined description of an automorphism $\alpha(\phi)\in\rm{Aut}(X)$, because for each vertex $y\in B_R(x)$ the diagram (\ref{xgcommute}) implies that $f(\alpha(\phi)(y))=\phi f(y)$; this formula, which is independent of $x$ and $h$, uniquely determines $\alpha(\phi)(y)$ since $f$ is injective on $V(X)$. And since $\alpha(\phi)$ agrees on $R$-balls with elements of $H$, we get that $\alpha(\phi)\in\mathscr{S}_R(H,X)$. It is also clear that $\alpha\circ\iota=\rm{id}$ on restriction to $\mathscr{S}_R(H,X)$.

We know from \cite{TwogenLeighton} that the theorem is true when $X=T$, so
there exists $g\in\mathscr{S}_M(\iota(H),T)$  such that $\iota(\Gamma_1)^g$ is commensurable to $\iota(\Gamma_2)$ in $\rm{Aut}(T)$. Therefore $\alpha\circ\iota(\Gamma_1)^{\alpha(g)}=\Gamma_1^{\alpha(g)}$ is commensurable to $\alpha\circ\iota(\Gamma_2)=\Gamma_2$ in $\rm{Aut}(X)$.
\end{proof}

\bigskip
As a consequence of Theorem \ref{t:qL2}, we obtain a version of Leighton's Theorem for simplicial complexes with free fundamental group.

\begin{cor}\label{cor:simplicial}
	Let $X_1$ and $X_2$ be finite simplicial complexes with free fundamental groups and a common universal cover. Then $X_1$ and $X_2$ have a common finite cover.
\end{cor}
\begin{proof}
	Let $\tilde{X}$ be the universal cover of $X_1$ and $X_2$, and let $\G_1,\G_2<\Aut(\tilde{X})$ be the corresponding deck-groups. To obtain a common finite cover of $X_1$ and $X_2$, we must find $g\in\Aut(\tilde{X})$ such that $\Gamma_1^g$ is commensurable to $\Gamma_2$. Since $\G_1$ and $\G_2$ are free groups, we see that $\tilde{X}$ is a quasitree. Let $Y$ be the quasitree graph obtained from the 1-skeleton of the barycentric subdivision of $\tilde{X}$ by attaching a leaf to each vertex of $\tilde{X}$ (to distinguish them from the new vertices of the barycentric subdivision). It is easy to check that there is a natural isomorphism $\Aut(Y)\cong\Aut(\tilde{X})$ of topological groups, and so the result follows by applying Theorem \ref{t:qL2} to $Y$.
\end{proof}

 \bigskip
\section{Leighton's Theorem fails for cocompact quasitrees}\label{s:cocompact}

Our purpose in this section is to prove the following special case of Theorem \ref{t:unimodularexample}. Our reasons
for proving this special case first were discussed in the introduction. A feature of the construction in this
proof is that the quasitree $X$ is cocompact (i.e.~there are only finitely many ${\rm{Aut}}(X)$-orbits of edges and
vertices) but ${\rm{Aut}}(X)$ does not contain a uniform lattice.
The proof ends with a
counting argument that draws inspiration from the following elementary 
argument based on \cite{BK} Example 4.12(1). 

\begin{example}\label{exmp:nouniformlattice}
Let $T$ be the simplicial tree whose vertex set $V(T)$ is divided into three infinite sets $V_2, V_3, V_4$
such that each $v\in V_i$ has valence $i$, 
no pair of adjacent vertices lies in the same $V_i$, 
each $v\in V_2$ is adjacent to one vertex in $V_3$ and one in $V_4$, and each vertex in $V_3\cup V_4$
is adjacent to a unique vertex in $V_2$. One can show that ${\rm{Aut}}(T)$ acts cocompactly on $T$
with 3-orbits of edges and three orbits of vertices (namely $V_2, V_3, V_4$). But $T$ does not
cover any finite graph, for if $Y$ were such a finite graph and $n_2, n_3, n_4$ were the number of vertices
of valence $2, 3, 4$ respectively, then by counting neighbours of vertices of valence $2$ we would have
$n_2=n_3=n_4$, whereas by counting valence-3 neighbours of valence-4 vertices we would have $2n_3 = 3n_4$.
\end{example}

\begin{thm}\label{thm:cocompactexample}
	There exists a cocompact quasitree graph $X$ and finite graphs $X_1$ and $X_2$ such that $X$ covers $X_1$ and $X_2$, but
	 there does not exist a finite graph $\hat{X}$ fitting into the following diagram of covers.
	 \begin{equation}\label{nohatX}
	\begin{tikzcd}[
	ar symbol/.style = {draw=none,"#1" description,sloped},
	isomorphic/.style = {ar symbol={\cong}},
	equals/.style = {ar symbol={=}},
	subset/.style = {ar symbol={\subset}}
	]
	&X\ar{d}\\
	&\hat{X}\ar{dl}\ar{dr}\\
	X_1&&X_2	
	\end{tikzcd}
	\end{equation}	
\end{thm}
\begin{proof} We will build $X_1$ and $X_2$ by augmenting a pair of finite graphs $Y_1$ and $Y_2$ 
for which there is a finite graph $Y$ and covering maps $p_i:Y\to Y_i$ for $i=1,2$.
	We will then build $X$ by assembling infinitely
	many copies of $Y$ in a tree-like manner. The following figure portrays the covering maps $p_i$,
	which are determined by giving three labels at each vertex of $Y_1$ and $Y_2$ describing the three preimages 
	of that vertex in $Y$.

	\begin{figure}[H]
		\centering	
		\scalebox{0.6}{
	\begin{tikzpicture}[auto,node distance=2cm,
	thick,every node/.style={circle,draw,font=\small},
	every loop/.style={min distance=2cm},
	hull/.style={draw=none},
	]
	\tikzstyle{label}=[draw=none,font=\Huge]

	\begin{scope}[scale=0.8]
	\node (6) {6};
	\node (5) [left of=6] {5};
	\node at (-2,4) (3a) {3a};
	\node at (-2,-4) (3b) {3b};
	\node at (2,2) (4a) {4a};
	\node at (2,-2) (4b) {4b};
	\node at (2,4) (2b) {2b};
	\node at (2,-4) (2a) {2a};
	\node at (4,0) (1) {1};
	\node[hull] (1Y) at (-5,-3) {};
	\node[hull] (2Y) at (5,-3) {};
	\node[label] (l) at (-6,2) {$Y$};
	
	\node[label] (p1) at (-8,-3.8) {$p_1$};
	\node[label] (p2) at (8,-3.8) {$p_2$};
	
	\path
	(3a) edge [bend right=50] (3b)
	(3a) edge [bend right=90] (3b)
	(3a) edge (5)
	(3b) edge (5)
	(5) edge [bend left] (6)
	(5) edge [bend right] (6)
	(3a) edge (2b)
	(3b) edge (2a)
	(6) edge (4a)
	(6) edge (4b)
	(4a) edge (4b)
	(4a) edge [bend left] (2b)
	(4a) edge [bend right] (2b)
	(4b) edge [bend left] (2a)
	(4b) edge [bend right] (2a)
	(1) edge [bend right] (2b)
	(1) edge [bend left] (2a)
	(1) edge [loop right] (1);	
	\end{scope}
	
	\begin{scope}[shift={(-10,-7)},scale=0.8]
	\node[ellipse] at (0,2) (2a) {2a,3a,4a};
	\node[ellipse] at (0,-2) (2b) {2b,3b,4b};
	\node[ellipse] at (2,0) (1) {1,5,6};	
	\node[hull] (Y1) at (1.5,2.5) {};
	\node[label] (l) at (-3,0) {$Y_1$};
	
	\path
	(2a) edge [bend left] (2b)
	(2a) edge [bend right] (2b)
	(2a) edge (2b)
	(2a) edge [bend left] (1)
	(2b) edge [bend right] (1)
	(1) edge [loop right](1);
	\end{scope}
	
	\begin{scope}[shift={(10,-7)},scale=0.8]
	\node[ellipse] at (0,2) (1) {1,3a,3b};
	\node[ellipse] at (0,0) (5) {2a,2b,5};
	\node[ellipse] at (0,-2) (6) {4a,4b,6};
	\node[hull] (Y2) at (-1.5,2.5) {};
	\node[label] (l) at (3,0) {$Y_2$};
	
	\path
	(1) edge [loop above] (1)
	(1) edge [bend left] (5)
	(1) edge [bend right] (5)
	(5) edge [bend left] (6)
	(5) edge [bend right] (6)
	(6) edge [loop below] (6);
	\end{scope}
	
	\draw[draw=black,fill=black,-triangle 90, ultra thick] (1Y) -- (Y1);
	\draw[draw=black,fill=black,-triangle 90, ultra thick] (2Y) -- (Y2);
	
	\end{tikzpicture}
}
\caption{\small The coverings $p_i:Y\to Y_i$.}
\end{figure}
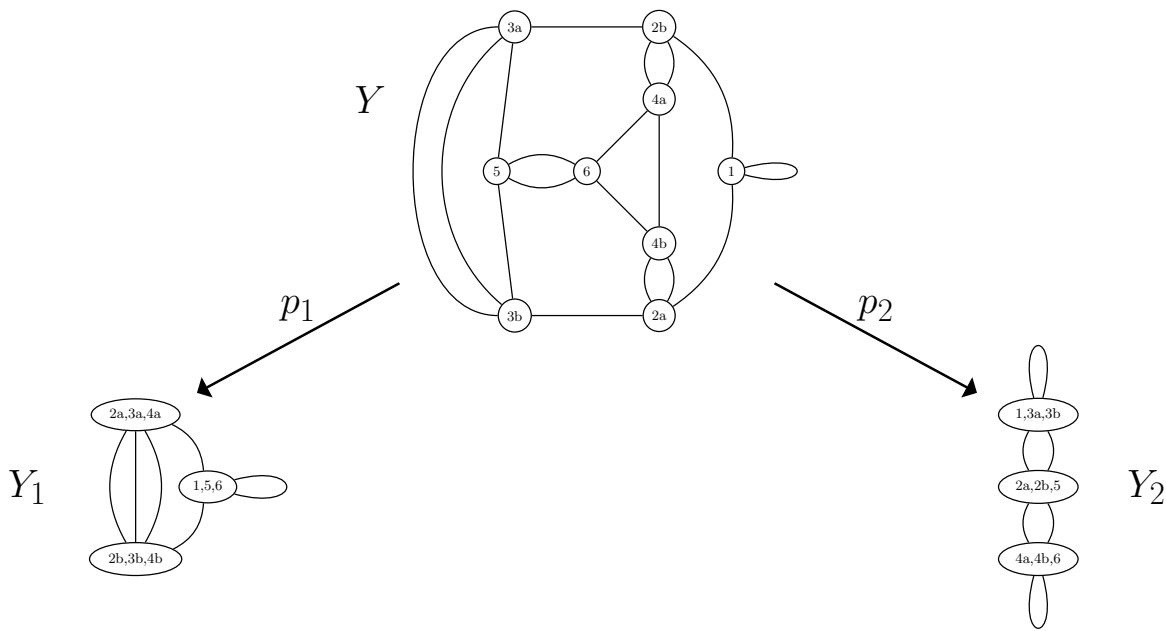
	
The graphs $X_1$ and $X_2$ are built from $Y_1$ and $Y_2$ by attaching both ends of a 
new subdivided edge to each vertex (i.e. each vertex of $Y_i$ has two edges emanating from it and they meet
at a valence-2 vertex of $X_i\ssm Y_i$). For ease of depiction in figures we will draw these subdivided edges as dashed edges, and refer to them as dashed edges for the rest of the proof. We will refer to the other edges as solid edges. $X$ is constructed by joining copies of the graph $Y$ by dashed edges in a tree-like structure as suggested by the following figure.
	\begin{figure}[H]
		\centering
	\scalebox{0.6}{
		\begin{tikzpicture}[auto,node distance=2cm,
		thick,every node/.style={circle,draw,font=\small},
		every edge/.style={thick,draw=blue,dashed},
		hull/.style={draw=none},
		scale=2
		]
		\tikzstyle{label}=[draw=none,font=\Huge]
		
		\node (1) {$Y$};
		\node (2) at (0,1) {$Y$};
		\node (3) at (1,1) {$Y$};
		\node (4) at (2,1) {$Y$};
		\node (5) at (0,2) {$Y$};
		\node (6) at (1,2) {$Y$};
		\node[label] at (3,1) {$\cdots$};
		\node[label] at (2,2) {$\cdots$};
		
		\node (7) at (0,-1) {$Y$};
		\node (8) at (1,-1) {$Y$};
		\node[label] at (2,-1) {$\cdots$};
		
		\node[label] at (-1,0) {$X$};
		
		\path
		(1) edge (2)
		(1) edge (3)
		(1) edge (4)
		(2) edge (5)
		(2) edge (6)
		(1) edge (7)
		(1) edge (8);
		
		\end{tikzpicture}
	}
\caption{\small Tree-like structure of $X$.}
\end{figure}
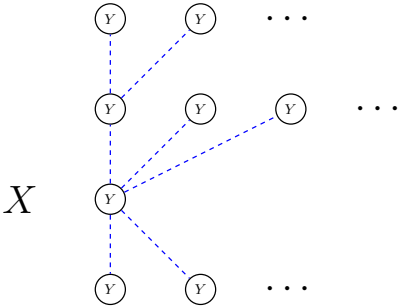

	 In more detail, we take an $18$-valent regular tree $T$, index copies of $Y$ by the vertices of this tree, and
	 at each vertex of each copy of $Y$ we add two dashed edges connecting to neighbouring copies of $Y$.
	  To complete the definition of $X$ (up to graph isomorphism) we must specify the labels of the vertices at the endpoints of each dashed edge. We do this so that for each bi-infinite line of dashed edges in $X$, the labels of the vertices will have one of the following six repeating patterns:
	\begin{align}\label{pattern}
	1,1,6,1,1,6\\\nonumber
	2a,2a,2a,2a,2a,2a\\\nonumber
	2b,2b,2b,2b,2b,2b\\\nonumber
	3a,4a,3a,4a,3a,4a\\\nonumber
	3b,4b,3b,4b,3b,4b\\\nonumber
	5,5,5,5,5,5\\\nonumber
	\end{align}
	Working out radially from a basepoint $\ast\in T$, we attach the copies of $Y$ in the next sphere
	by dashed edges
	so as to continue the $6$ specified patterns in the forwards and backwards directions. No pair of copies of $Y$
	is joined by more than one edge.
	At the first stage, the copy $Y_\ast$ of $Y$ at $\ast$ has its vertex labelled $1$ attached to an adjacent
	copy of $Y$ by a dashed edge ending at a vertex with label $1$, and to another adjacent copy with an
	edge ending at a vertex labelled $6$; the vertex of $Y_\ast$ labelled $2a$ is connected to vertices with the same
	label in two other neighbouring copies of $Y$; likewise the vertices with label $2b$ and $5$; while the vertex of $Y_\ast$ with
	label $3a$ is connected to two neighbours at their vertices labelled $4a$, and the vertex of $Y_\ast$ with
	label $3b$ is connected to the two remaining neighbours at their vertices labelled $4b$. At subsequent stages,
	as one works away from the basepoint $\ast\in T$, at $v\in V(T)$ one first deals with the vertex of $Y_v$ that is already connected to a copy of $Y$ one-step closer to $\ast$ than $v$, connecting that vertex to a copy of $Y$ in the next generation by a dashed
	edge ending at a vertex that continues
	the pattern of the word labelling the dashed path from $\ast$ to $v$; the remaining vertices of $Y_v$ are then connected to the
	remaining $16$ neighbouring copies of $Y$ by the same prescription used at the basepoint. 
	
	It follows from this constructive description of $X$ that any copy of $Y$ in $X$ can be sent to any other copy by a
	suitable automorphism of $X$: send the first copy of $Y$ to the second by the unique label-preserving map, then
	proceed radially outwards through the template tree $T$ to extend the map to dashed edges and
	neighbouring copies of $Y$ in the unique label-preserving manner. In particular the quasitree $X$ is cocompact.

	Our next task is to describe covering maps from $X$ to $X_1$ and $X_2$. To this end, observe that each of our six repeating
	patterns of labels (\ref{pattern}) stays within the set of labels of one of the vertices of $Y_1$. This enables us to define a covering $X\to X_1$ by requiring that it restrict to $p_1:Y\to Y_1$ on each copy of $Y$ in $X$. (There is an arbitrary choice to be made
	in orientation of how each dashed edge then maps to $X_1$, but this is of no consequence.) In the case of $Y_2$,
	 the middle vertex has all the labels for the second, third and sixth patterns in (\ref{pattern}), but the other three patterns have labels that switch between the top and bottom vertices of $Y_2$, so we have to twist $p_2$ appropriately to account for this.
We fix an automorphism $\tau:Y_2\to Y_2$ of order $2$ that interchanges the top and bottom vertices, then define a covering $X\to X_2$ 
that restricts to either $p_2:Y\to Y_2$ or $\tau p_2:Y\to Y_2$ on each copy of $Y$ in $X$; to determine whether $p_2$ or
$\tau p_2$ is applied to a particular copy of $Y$, we again work out radially from the basepoint $\ast$ in the template tree $T$;
first we map $Y_\ast$ by $p_2$; subsequently,
each copy $Y_\lambda$ of $Y$ is attached to a unique copy $Y_{\lambda'}$ of $Y$ closer to $\ast$, and the restriction
$\pi_\lambda: Y_\lambda\to Y_2$ of $X\to X_2$ is defined to be $\pi_{\lambda'}$ or $\tau\pi_{\lambda'}$ according to whether 
the vertices of $Y_\lambda$ and $Y_{\lambda'}$ that are connected by a dashed edge have labels at the same vertex of $Y_2$ or not.
Having defined $X\to X_2$ on the disjoint union of the copies of $Y$, we again complete the description of the covering map
by arbitrarily choosing an orientation for the image of each dashed edge.

To complete the proof of the theorem,
we must argue that there is no finite graph $\hat{X}$ that fits into a diagram of covers (\ref{nohatX}).
If $\hat{X}$ did exist, it would consist of solid and dashed edges in the same way as $X$, and each 
connected component $\hat{Y}$ of the solid subgraph in $\hat{X}$
would fit into the following diagram of covers: 
\begin{equation}\label{nohatY}
\begin{tikzcd}[
ar symbol/.style = {draw=none,"#1" description,sloped},
isomorphic/.style = {ar symbol={\cong}},
equals/.style = {ar symbol={=}},
subset/.style = {ar symbol={\subset}}
]
&Y\ar{d}\\
&\hat{Y}\ar{dl}\ar{dr}\\
Y_1&&Y_2	
\end{tikzcd}
\end{equation}
But the covers from $Y$ to $Y_1$ and $Y_2$ are degree three, and $Y_1$ is not isomorphic to $Y_2$, so this forces $Y\to\hat{Y}$ to be an isomorphism. Thus $\hat{X}$ could be constructed from copies of $Y$ connected by dashed edges. One can easily verify that the numbers in the labels of vertices of $Y$ correspond to orbits of ${\rm{Aut}}(Y)$ (for example the vertex $1$ is fixed by any automorphism
because it is the base of the only loop of length one). 
It follows that the vertices of $\hat{X}$ inherit a labelling by numbers from $Y$ (ignoring the letter suffix of labels)
such that the putative covering $X\to \hat{X}$ preserves the numbering of vertices. 
Therefore, 
the bi-infinite lines of dashed edges in $X$, which have
itineraries described by the patterns (\ref{pattern}) must cover circles of dashed edges in $\hat{X}$ 
whose itineraries (cyclic orderings of numerical vertex labels) follow the same pattern. This leads to a contradiction
as follows:
each vertex of $\hat{X}$ labelled $1$ or $6$ lies on a unique dashed circle (since there is a unique dashed line
through each of its pre-images in $X$), and the itinerary of this dashed circle is a repetition of the cyclic pattern
$1,1,6,1,1,6$; this forces $\hat{X}$ to contain twice as many vertices labelled $1$ versus vertices labelled $6$; this
contradicts the fact that each copy of $Y=\hat{Y}$ in $\hat{X}$ has one vertex labelled $1$ and one labelled $6$. 
\end{proof}

\begin{remk}\label{rem:notuni}
	The above proof in fact shows that $\rm{Aut}(X)$ does not contain a cocompact lattice;
	equivalently, there is no {\em regular} covering $X\to Z$ with $Z$ a finite graph. 
	In more detail, what the argument shows is that $X$ does not cover any finite graph 
	constructed from (solid) copies of $Y$ by adding dashed circles so that one such circle passes through each vertex.
	Any uniform lattice in $\rm{Aut}(X)$ would be virtually free,
	and the quotient of $X$ by a free subgroup $\Lambda$ of finite index would be a finite graph assembled from copies of $Y$
	in this manner. Indeed the action of $\rm{Aut}(X)$ permutes the obvious copies of $Y$, and the non-trivial elements
	of $\Lambda$ move each copy of $Y$ off itself. 
\end{remk}

\bigskip
\section{Leighton's Theorem can fail for  quasitrees even if ${\rm{Aut}}(X)$ contains a uniform lattice}\label{s:unimodular}

Theorem \ref{t:qL} concerns regular coverings by a quasitree, while in the contrasting
result Theorem \ref{thm:cocompactexample} the coverings are not regular
and the automorphism group of the quasitree does not contain a uniform lattice (see Remark \ref{rem:notuni}). This leaves open the question of 
whether one can exploit the irregularity of covers to prove a version of  
Theorem \ref{thm:cocompactexample} in which the automorphism group of the quasitree does contain a
uniform lattice. Our purpose in this
section is to settle this question by proving Theorem \ref{t:unimodularexample} from the introduction.
We repeat the statement for the reader's convenience.

\begin{thm}\label{t:unimodularexample2}
There exists a locally finite quasitree $X$, finite graphs $X_1$ and $X_2$, and
covering maps
$X\to X_1$ and $X\to X_2$,  with $X\to X_2$ regular, for which there is no finite graph $\hat{X}$ fitting into the following diagram of covers.
	 \begin{equation}\label{nohatX2}
	\begin{tikzcd}[
	ar symbol/.style = {draw=none,"#1" description,sloped},
	isomorphic/.style = {ar symbol={\cong}},
	equals/.style = {ar symbol={=}},
	subset/.style = {ar symbol={\subset}}
	]
	&X\ar{d}\\
	&\hat{X}\ar{dl}\ar{dr}\\
	X_1&&X_2	
	\end{tikzcd}
	\end{equation}	
\end{thm}

\subsection{The graphs $X_1, X_2$ and $X$}
As in the proof of Theorem \ref{thm:cocompactexample}, we will build $X$ by assembling infinitely many copies of a finite graph $Y$ 
in a tree-like pattern. We shall again imagine these copies of $Y$ being made out of solid edges, with the edges joining them being dashed, and all graph morphisms we consider must map solid edges to solid edges and dashed to dashed.
Formally speaking we subdivide the dashed edges to distinguish them from the solid edges, but the arguments (and diagrams) run more smoothly if we think of them as being dashed.
Once again, $X_1$ and $X_2$ will be obtained from finite graphs $Y_1$ and $Y_2$ that are covered by $Y$. 
This time we take $Y_2=Y$, so there is only one non-trivial covering map $p_1:Y\to Y_1$ to be considered; this
is portrayed in the following figure, where we again label vertices to encode the map. 
\\

\begin{figure}[H]
	\centering
\scalebox{0.6}{
\begin{tikzpicture}[auto,node distance=2cm,
thick,every node/.style={circle,draw,font=\small,minimum size=0.8cm},
every loop/.style={min distance=2cm},
hull/.style={draw=none},
]
\tikzstyle{label}=[draw=none,font=\Huge]

\begin{scope}[scale=2]
\node (1) {a};
\node (2) at (-1,1) {a};
\node (3) at (1,1) {b};
\node (4) at (-1,-1) {a};
\node (5) at (1,-1) {c};
\node (6) at (-2,2) {b};
\node (7) at (2,2) {c};
\node (8) at (-2,-2) {c};
\node (9) at (2,-2) {b};
\node[hull] (Y) at (0,-2.2) {};
\node[label] (l) at (-3,0) {$Y$};

\node[label] (p) at (-0.4,-2.7) {$p_1$};

\path
(1) edge (2)
(1) edge (3)
(1) edge (4)
(1) edge (5)
(2) edge (4)
(3) edge (5)
(2) edge (6)
(3) edge (7)
(4) edge (8)
(5) edge (9)
(6) edge [bend left=10] (8)
(6) edge [bend right=10] (8)
(7) edge [bend left=10] (9)
(7) edge [bend right=10] (9)
(6) edge [bend left=30] (5)
(8) edge [bend right=30] (3)
(7) edge [bend right=30] (2)
(9) edge [bend left=30] (4);
\end{scope}

\begin{scope}[shift={(0,-10)},scale=1.5]

\node (1) {a};
\node (2) at (2,2) {b};
\node (3) at (2,-2) {c};	
\node[hull] (Y1) at (0,1.5) {};
\node[label] (l) at (-4,0) {$Y_1$};

\path
(1) edge [loop left] (1)
(1) edge (2)
(1) edge (3)
(2) edge (3)
(2) edge [bend left=20] (3)
(2) edge [bend right=20] (3);

\end{scope}

\draw[draw=black,fill=blue,-triangle 90, ultra thick] (Y) -- (Y1);

\end{tikzpicture}
}
\caption{\small The covering $p_1:Y\to Y_1$.}\label{fig:p1}
\end{figure}
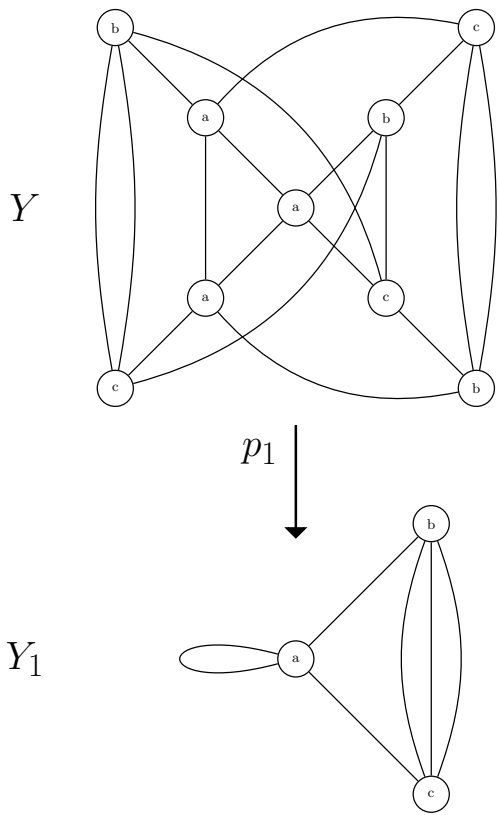

Let $\mathcal{P}$ be the collection of all covering maps $Y\to Y_1$ (not just those
two that respect the labelling of Figure \ref{fig:p1}). For each vertex $v\in V(Y)$ we want to determine the set $\{p(v)\mid p\in\mathcal{P}\}$. 
For any $p\in\mathcal{P}$, the loop of length $1$ in $Y_1$ must lift to a collection of cycles
in $Y$ of total length $3$. Since there are no 1-cycles or 2-cycles in $Y$, it must lift to a 3-cycle, and there are only two of those, namely the two triangles containing the central vertex, which we shall henceforth refer to as the \emph{left triangle} and \emph{right triangle}. Put another way:
\begin{equation}\label{tri-a}
{\text{Every covering map $Y\to Y_1$ maps all of the vertices of either the left or right trinagle onto $a$}.}
\end{equation}
 Therefore the only vertices in $Y$ that can map to the vertex $a$ in $Y_1$ are the five central vertices, and the central vertex must map to vertex $a$. It turns out that these are the only restrictions; in the following diagram we label the vertices $v\in Y$ with the sets $\{p(v)\mid p\in\mathcal{P}\}$.
\begin{figure}[H]
	\centering
\scalebox{0.6}{
\begin{tikzpicture}[auto,node distance=2cm,
thick,every node/.style={circle,draw,font=\small,minimum size=0.8cm},
every loop/.style={min distance=2cm},
hull/.style={draw=none},
]
\tikzstyle{label}=[draw=none,font=\Huge]

\begin{scope}[scale=2]
\node (1) {a};
\node (2) at (-1,1) {abc};
\node (3) at (1,1) {abc};
\node (4) at (-1,-1) {abc};
\node (5) at (1,-1) {abc};
\node (6) at (-2,2) {bc};
\node (7) at (2,2) {bc};
\node (8) at (-2,-2) {bc};
\node (9) at (2,-2) {bc};
\node[hull] (Y) at (0,-2.2) {};
\node[label] (l) at (-3,0) {$Y$};

\path
(1) edge (2)
(1) edge (3)
(1) edge (4)
(1) edge (5)
(2) edge (4)
(3) edge (5)
(2) edge (6)
(3) edge (7)
(4) edge (8)
(5) edge (9)
(6) edge [bend left=10] (8)
(6) edge [bend right=10] (8)
(7) edge [bend left=10] (9)
(7) edge [bend right=10] (9)
(6) edge [bend left=30] (5)
(8) edge [bend right=30] (3)
(7) edge [bend right=30] (2)
(9) edge [bend left=30] (4);
\end{scope}

\begin{scope}[shift={(0,-10)},scale=1.5]

\node (1) {a};
\node (2) at (2,2) {b};
\node (3) at (2,-2) {c};	
\node[hull] (Y1) at (0,1.5) {};
\node[label] (l) at (-4,0) {$Y_1$};

\path
(1) edge [loop left] (1)
(1) edge (2)
(1) edge (3)
(2) edge (3)
(2) edge [bend left=20] (3)
(2) edge [bend right=20] (3);

\end{scope}

\draw[draw=black,fill=blue,-triangle 90, ultra thick] (Y) -- (Y1);

\end{tikzpicture}
}
\caption{\small Vertices $v\in Y$ labelled by the sets $\{p(v)\mid p\in\mathcal{P}\}$.}\label{fig:PlabelY}
\end{figure}
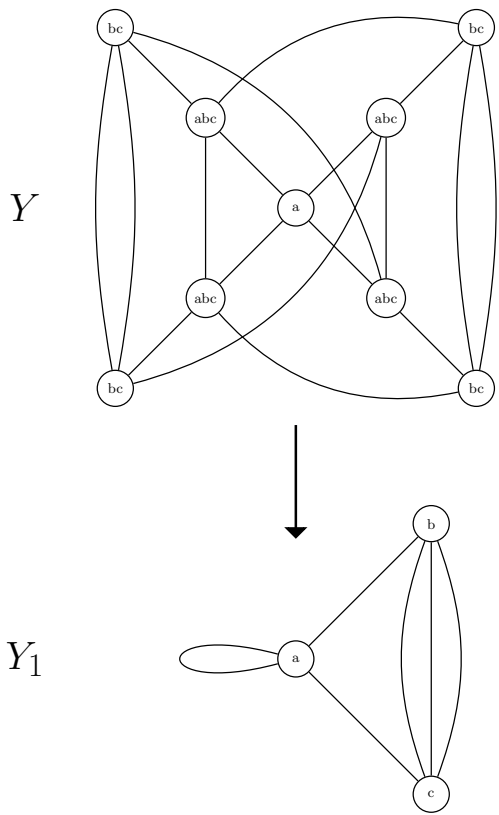

To see that each possibility for $p$ can be obtained, it is enough to compose $p_1$ with symmetries of $Y$:
the automorphism group of $Y$ acts transitively on the set of four outside vertices and on the set of vertices
labelled $abc$.

The graphs $X_1$ and $X_2$ are obtained from $Y_1$ and $Y$ by attaching dashed edges as follows (the dashed edges are blue in the figures for extra clarity). Throughout 
this section, we will implicitly make the identifications $V(X_1)=V(Y_1)$ and $V(X_2)=V(Y)$.

\begin{figure}[H]
	\centering
\scalebox{0.6}{
\begin{tikzpicture}[auto,node distance=2cm,
thick,every node/.style={circle,draw,font=\small,minimum size=0.8cm},
every loop/.style={min distance=2cm},
hull/.style={draw=none},
]
\tikzstyle{label}=[draw=none,font=\Huge]

\begin{scope}[shift={(-7,0)},scale=1.5]
\node (1) {a};
\node (2) at (2,2) {b};
\node (3) at (2,-2) {c};
\node[label] (l) at (1,3.6) {$X_1$};	
\node[hull] (Y1) at (0,2.5) {};

\path
(1) edge [loop left] (1)
(1) edge (2)
(1) edge (3)
(2) edge (3)
(2) edge [bend left=20] (3)
(2) edge [bend right=20] (3)

(1) edge [blue,dashed, bend left=30] (2)
(1) edge [blue,dashed, bend left=40] (2)
(2) edge [blue,dashed, bend left=40] (3)
(2) edge [blue,dashed, bend left=50] (3)
(3) edge [blue,dashed, bend left=30] (1)
(3) edge [blue,dashed, bend left=40] (1);
\end{scope}

\begin{scope}[shift={(7,0)},scale=2]
\node (1) {a};
\node (2) at (-1,1) {abc};
\node (3) at (1,1) {abc};
\node (4) at (-1,-1) {abc};
\node (5) at (1,-1) {abc};
\node (6) at (-2,2) {bc};
\node (7) at (2,2) {bc};
\node (8) at (-2,-2) {bc};
\node (9) at (2,-2) {bc};

\node[label] (l) at (0,2.7) {$X_2$};

\path
(1) edge (2)
(1) edge (3)
(1) edge (4)
(1) edge (5)
(2) edge (4)
(3) edge (5)
(2) edge (6)
(3) edge (7)
(4) edge (8)
(5) edge (9)
(6) edge [bend left=10] (8)
(6) edge [bend right=10] (8)
(7) edge [bend left=10] (9)
(7) edge [bend right=10] (9)
(6) edge [bend left=30] (5)
(8) edge [bend right=30] (3)
(7) edge [bend right=30] (2)
(9) edge [bend left=30] (4)

(6) edge [blue,dashed,bend right=30] (8)
(6) edge [blue,dashed,bend right=40] (8)
(6) edge [blue,dashed,bend left=20] (1)
(6) edge [blue,dashed,bend left=30] (1)
(1) edge [blue,dashed,bend left=20] (8)
(1) edge [blue,dashed,bend left=30] (8)

(4) edge [blue,dashed,bend left=15] (2)
(4) edge [blue,dashed,bend left=25] (2)
(4) edge [blue,dashed,bend left=35] (2)
(2) edge [blue,dashed,bend left=20] (7)
(9) edge [blue,dashed,bend left=20] (4)
(7) edge [blue,dashed,bend left=20] (3)
(7) edge [blue,dashed,bend right=20] (3)
(5) edge [blue,dashed,bend left=20] (9)
(5) edge [blue,dashed,bend right=20] (9)
(5) edge [blue,dashed,bend right=25] (3)
(5) edge [blue,dashed,bend right=35] (3) 
(9) edge [blue,dashed,bend right=40] (7);
\end{scope}

\end{tikzpicture}
}
\caption{\small The graphs $X_1$ and $X_2$.}\label{fig:X12}
\end{figure}
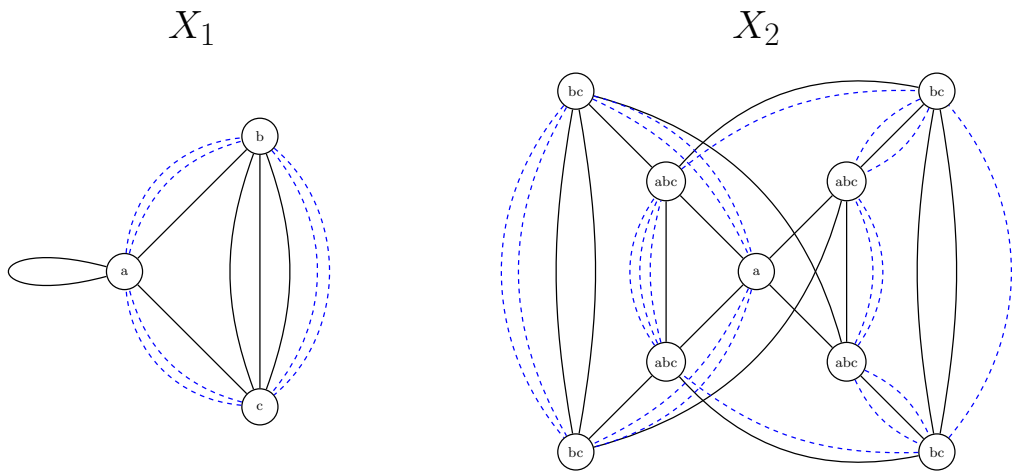

We define the graph $X$ by taking the regular covering $q_2:X\to X_2$ corresponding to the normal subgroup
$\<\!\<\pi_1Y_2\>\!\>$ of $\pi_1 X_2$; the Galois group $\G$ of this covering is free of rank $18$. Note that $\G<{\rm{Aut}}(X)$ is a uniform lattice. 
$X$ is (up to isomorphism) the unique regular cover of $X_2$ that has no cycle containing a dashed edge and is such that each
connected component of the solid subgraph is isomorphic to $Y$. 
In particular $X$ can be regarded as infinitely many copies of $Y$ assembled according to a tree-like template, 
as in the proof of Theorem \ref{thm:cocompactexample}, and we will again use this template to build maps by working
outwards from a basepoint in a radial manner.

\subsection{Extendible edges and maps}
We want to prove that there is a 
covering map $q_1:X\to X_1$. This covering map will restrict to a covering $p:Y\to Y_1$ on each copy of $Y$ in $X$.
We have already considered limitations on where $p$ can send vertices, and we have to work with these constraints
as we determine where the dashed edges of $X$ map: each dashed edge in $X$ must map to a dashed edge in $X_1$ in such a way that we can extend continuously to coverings $Y\to Y_1$ on adjacent copies of $Y$. With this in mind we make the following definition, where we work with oriented edges, writing $\varepsilon(0)$ and $\varepsilon(1)$ to denote the initial and terminal vertices of an edge $\varepsilon\in E^{\pm}(X_i)$.
This definition is made more awkward by the need to exclude an {\em exceptional set} $\mathcal{E}$ of pairs
of edges; by definition $\mathcal{E}$ is any pair of edges $(\varepsilon_1,\varepsilon_2)\in E^{\pm}(X_1)\times E^{\pm}(X_2)$ that are shown as red and thick-dashed in Figure \ref{fig:exceptional} (with any orientations).

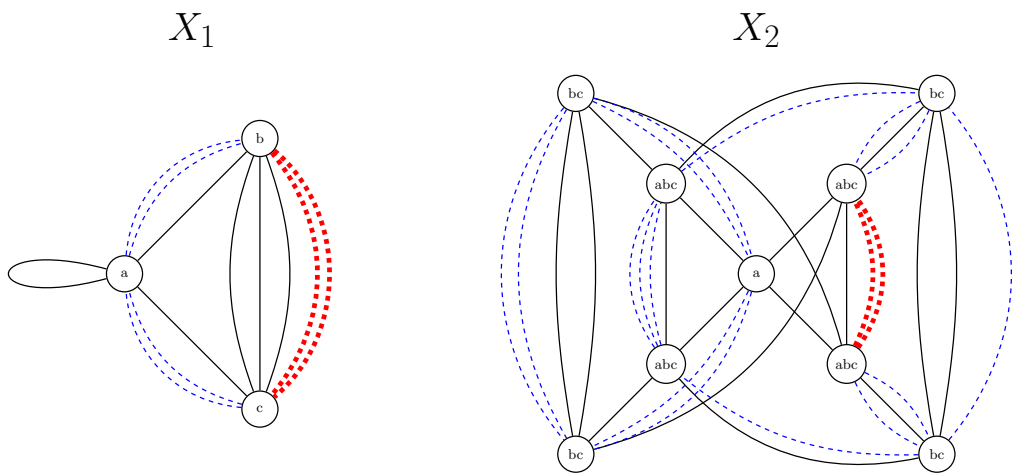
\begin{figure}[H]
	\centering
	\scalebox{0.6}{
		\begin{tikzpicture}[auto,node distance=2cm,
		thick,every node/.style={circle,draw,font=\small,minimum size=0.8cm},
		every loop/.style={min distance=2cm},
		hull/.style={draw=none},
		]
		\tikzstyle{label}=[draw=none,font=\Huge]
		
		\begin{scope}[shift={(-7,0)},scale=1.5]
		\node (1) {a};
		\node (2) at (2,2) {b};
		\node (3) at (2,-2) {c};	
		\node[hull] (Y1) at (0,2.5) {};
		\node[label] (l) at (1,3.6) {$X_1$};
		
		\path
		(1) edge [loop left] (1)
		(1) edge (2)
		(1) edge (3)
		(2) edge (3)
		(2) edge [bend left=20] (3)
		(2) edge [bend right=20] (3)
		
		(1) edge [blue,dashed, bend left=30] (2)
		(1) edge [blue,dashed, bend left=40] (2)
		(2) edge [red,dashed,line width=3pt, bend left=40] (3)
		(2) edge [red,dashed,line width=3pt, bend left=50] (3)
		(3) edge [blue,dashed, bend left=30] (1)
		(3) edge [blue,dashed, bend left=40] (1);
		\end{scope}
	
	\begin{scope}[shift={(7,0)},scale=2]
		\node (1) {a};
		\node (2) at (-1,1) {abc};
		\node (3) at (1,1) {abc};
		\node (4) at (-1,-1) {abc};
		\node (5) at (1,-1) {abc};
		\node (6) at (-2,2) {bc};
		\node (7) at (2,2) {bc};
		\node (8) at (-2,-2) {bc};
		\node (9) at (2,-2) {bc};
		\node[label] (X2) at (0,2.7) {$X_2$};
		
		\path
		(1) edge (2)
		(1) edge (3)
		(1) edge (4)
		(1) edge (5)
		(2) edge (4)
		(3) edge (5)
		(2) edge (6)
		(3) edge (7)
		(4) edge (8)
		(5) edge (9)
		(6) edge [bend left=10] (8)
		(6) edge [bend right=10] (8)
		(7) edge [bend left=10] (9)
		(7) edge [bend right=10] (9)
		(6) edge [bend left=30] (5)
		(8) edge [bend right=30] (3)
		(7) edge [bend right=30] (2)
		(9) edge [bend left=30] (4)
		
		(6) edge [blue,dashed,bend right=30] (8)
		(6) edge [blue,dashed,bend right=40] (8)
		(6) edge [blue,dashed,bend left=20] (1)
		(6) edge [blue,dashed,bend left=30] (1)
		(1) edge [blue,dashed,bend left=20] (8)
		(1) edge [blue,dashed,bend left=30] (8)
		
		(4) edge [blue,dashed,bend left=15] (2)
		(4) edge [blue,dashed,bend left=25] (2)
		(4) edge [blue,dashed,bend left=35] (2)
		(2) edge [blue,dashed,bend left=20] (7)
		(9) edge [blue,dashed,bend left=20] (4)
		(7) edge [blue,dashed,bend left=20] (3)
		(7) edge [blue,dashed,bend right=20] (3)
		(5) edge [blue,dashed,bend left=20] (9)
		(5) edge [blue,dashed,bend right=20] (9)
		(5) edge [red,dashed,line width=3pt,bend right=25] (3)
		(5) edge [red,dashed,line width=3pt,bend right=35] (3)
		
		(9) edge [blue,dashed,bend right=40] (7);
		\end{scope}
		
		\end{tikzpicture}
	}
	
	\caption{\small The red thick-dashed pairs in the exceptional set}\label{fig:exceptional}
\end{figure} 

We define a pair of oriented dashed edges $(\vare_1,\vare_2)\in E^{\pm}(X_1)\times E^{\pm}(X_2)$
 to be \emph{extendible} if $(\vare_1,\vare_2)\notin\mathcal{E}$ and
there exist $p,p'\in\mathcal{P}$ such that $p(\vare_2(0))=\vare_1(0)$ and $p'(\vare_2(1))=\vare_1(2)$. Extendible pairs can easily be read off from Figure \ref{fig:exceptional}, because
the condition just requires that the label of $\vare_1(0)$ is an element of the set of labels of $\vare_2(0)$ and 
that the label of $\vare_1(1)$ is an element of the set of labels at $\vare_2(1)$. Note that $(\e_1,\e_2)$ is extendible if and only if the pair with reversed orientations $(\bar{\e}_1,\bar{\e}_2)$ is extendible.

Given a vertex $v$ let $\lkblu (v)$ denote the set of oriented dashed edges with initial vertex $v$. If $v\in V(X_2)$ and $p\in\mathcal{P}$, we say that a bijection $\sigma:\lkblu (v)\to\lkblu (p(v))$ is \emph{extendible} if every pair $(\sigma(\e),\e)$ is extendible.\\

\begin{lemma}\label{claims12}
\begin{enumerate}[(1)]
\item For any $p\in\mathcal{P}$ and vertex $v\in V(X_2)$, there exists an extendible bijection $\sigma:\lkblu (v)\to\lkblu (p(v))$.
\item  For any $p\in\mathcal{P}$, vertex $v\in V(X_2)$ and extendible pair $(\e_1,\e_2)\in\lkblu (p(v))\times\lkblu (v)$, there exists an extendible bijection $\sigma:\lkblu (v)\to\lkblu (p(v))$ such that $\sigma(\e_2)=\e_1$.
\end{enumerate}
\end{lemma} 

\begin{proof}
To prove these claims one must consider every pair of vertices $(p(v),v)$ with $v\in V(X_2)$ and $p\in\mathcal{P}$, and inspect the labellings on the dashed edges in the links of these vertices. We will do this explicitly for the particular
pair of vertices highlighted in Figure \ref{fig:sigmaanalysis} below;  the arguments for other pairs are similar, but some
remarks will be needed concerning the exceptional set $\mathcal{E}$.
We have coloured various edges in Figure \ref{fig:sigmaanalysis} so that we can refer to them
accurately without getting overrun by notation.
 
To begin, note that each yellow edge in $\lkblu (v)$
ends at a vertex labelled $abc$, so an extendible $\sigma$ can map these edges to either the green or orange edges. The red edge goes to a vertex labelled $bc$, so to be extendible $\sigma$ must map this red edge to a green one;
these constraints can obviously be satisfied, so claim (1), the existence of extendible maps $\sigma$,
is established in this case. 
Assertion (2) is also clear in this case: if we specify that the red edge maps to one of the green edges then any extension of $\sigma$ to the yellow edges will be extendible, whilst if we specify where one of the yellow edges goes then this leaves at least one green edge available for the red edge to map to; so in either case, we can extend the initial assignment to an extendible $\sigma$.\\

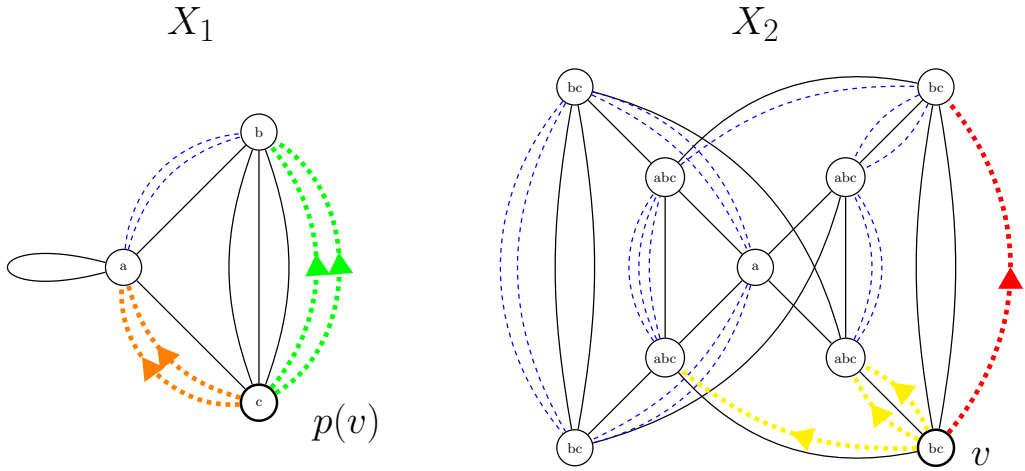
\begin{figure}[H]
	\centering
\scalebox{0.6}{
\begin{tikzpicture}[auto,node distance=2cm,
thick,every node/.style={circle,draw,font=\small,minimum size=0.8cm},
every loop/.style={min distance=2cm},
hull/.style={draw=none}
]
\tikzstyle{label}=[draw=none,font=\Huge]

\begin{scope}[shift={(-7,0)},scale=1.5]
\node (1) {a};
\node (2) at (2,2) {b};
\node (3) [ultra thick] at (2,-2) {c};	
\node[hull] (Y1) at (0,2.5) {};
\node[label] (l) at (1,3.6) {$X_1$};
\node[label] at (3.3,-2.3) {$p(v)$};

\path
(1) edge [loop left] (1)
(1) edge (2)
(1) edge (3)
(2) edge (3)
(2) edge [bend left=20] (3)
(2) edge [bend right=20] (3)

(1) edge [blue,dashed, bend left=30] (2)
(1) edge [blue,dashed, bend left=40] (2)
(3) edge [green,dashed,line width=3pt, bend right=40,
postaction={decoration={markings,mark=at position 0.55 with {\arrow[green,line width=1mm]{triangle 60}}},decorate}] (2)
(3) edge [green,dashed,line width=3pt, bend right=60,
postaction={decoration={markings,mark=at position 0.55 with {\arrow[green,line width=1mm]{triangle 60}}},decorate}] (2)
(3) edge [orange,dashed,line width=3pt, bend left=30,
postaction={decoration={markings,mark=at position 0.6 with {\arrow[orange,line width=1mm]{triangle 60}}},decorate}] (1)
(3) edge [orange,dashed,line width=3pt, bend left=50,
postaction={decoration={markings,mark=at position 0.6 with {\arrow[orange,line width=1mm]{triangle 60}}},decorate}] (1);
\end{scope}

\begin{scope}[shift={(7,0)},scale=2]
	
\node (1) {a};
\node (2) at (-1,1) {abc};
\node (3) at (1,1) {abc};
\node (4) at (-1,-1) {abc};
\node (5) at (1,-1) {abc};
\node (6) at (-2,2) {bc};
\node (7) at (2,2) {bc};
\node (8) at (-2,-2) {bc};
\node (9) [ultra thick] at (2,-2) {bc};
\node[label] (X2) at (0,2.7) {$X_2$};
\node[label] at (2.5,-2.1) {$v$};

\path
(1) edge (2)
(1) edge (3)
(1) edge (4)
(1) edge (5)
(2) edge (4)
(3) edge (5)
(2) edge (6)
(3) edge (7)
(4) edge (8)
(5) edge (9)
(6) edge [bend left=10] (8)
(6) edge [bend right=10] (8)
(7) edge [bend left=10] (9)
(7) edge [bend right=10] (9)
(6) edge [bend left=30] (5)
(8) edge [bend right=30] (3)
(7) edge [bend right=30] (2)
(9) edge [bend left=30] (4)

(6) edge [blue,dashed,bend right=30] (8)
(6) edge [blue,dashed,bend right=40] (8)
(6) edge [blue,dashed,bend left=20] (1)
(6) edge [blue,dashed,bend left=30] (1)
(1) edge [blue,dashed,bend left=20] (8)
(1) edge [blue,dashed,bend left=30] (8)

(4) edge [blue,dashed,bend left=15] (2)
(4) edge [blue,dashed,bend left=25] (2)
(4) edge [blue,dashed,bend left=35] (2)
(2) edge [blue,dashed,bend left=20] (7)
(9) edge [yellow,dashed,line width=3pt,
postaction={decoration={markings,mark=at position 0.5 with {\arrow[yellow,line width=1mm]{triangle 60}}},decorate},bend left=20] (4)
(7) edge [blue,dashed,bend left=20] (3)
(7) edge [blue,dashed,bend right=20] (3)
(9) edge [yellow,dashed,line width=3pt,
postaction={decoration={markings,mark=at position 0.65 with {\arrow[yellow,line width=1mm]{triangle 60}}},decorate},bend right=20] (5)
(9) edge [yellow,dashed,line width=3pt,
postaction={decoration={markings,mark=at position 0.65 with {\arrow[yellow,line width=1mm]{triangle 60}}},decorate},bend left=20] (5)
(5) edge [blue,dashed,bend right=25] (3)
(5) edge [blue,dashed,bend right=35] (3)

(9) edge [red,dashed,line width=3pt,
postaction={decoration={markings,mark=at position 0.5 with {\arrow[red,line width=1mm]{triangle 60}}},decorate},bend right=40] (7);
\end{scope}

\end{tikzpicture}
}
\caption{\small Analysis of maps $\sigma:\lkblu (v)\to\lkblu (p(v))$ for a particular pair $(p(v),v)$.}\label{fig:sigmaanalysis}
\end{figure}

For the other pairs $(p(v),v)$ the proof is equally straightforward, having excluded
the troublesome edges from Figure \ref{fig:exceptional}. If $v$ and $p(v)$ are vertices incident at the thick-dashed red edges, then any pair of these red edges would be extendible if we had not explicitly excluded the set $\mathcal{E}$, but (2) would fail because there is no
extendible bijection $\sigma:\lkblu (v)\to\lkblu (p(v))$ mapping a red edge to a red edge, since then one of the blue edges in $\lkblu(v)$ would have to map to a blue edge in $\lkblu(p(v))$.
We removed this problem by fiat.
\end{proof}

\subsection{The covering $X\to X_1$}
For $Z$ a subgraph of $X$, we say that an immersion (locally injective graph morphism) $q:Z\to X_1$ is \emph{extendible} if the pair $(q(\e),q_2(\e))$ is extendible for every dashed edge $\e$ in $Z$. We construct the cover $q_1:X\to X_1$ 
by working outwards from a base copy of $Y$ inductively using the following lemma. Note that defining $q_1$ on the base copy of $Y$ and the dashed edges that meet it also follows from the lemma by setting $Z=\emptyset$.\\

\begin{lemma}
Suppose $Z$ is a finite connected subgraph of $X$ consisting of a number of (solid) 
copies of $Y$ and all of the dashed edges that meet them, and suppose $q:Z\to X_1$ is extendible. If $Y'\subset X$ is a copy of $Y$ adjacent to $Z$, then we can extend $q$ to an extendible map $q':Z'\to X_1$, where $Z'$ is the union of $Z$ with $Y'$ and any dashed edges that meet it.
\end{lemma}

\begin{proof}
$Y'$ is identified with $Y$ via the covering $p_2:X\to X_2$. We are forced to put $q'=q$ on $Z$, so it remains to define $q$ on $Y'$ and the dashed edges that meet it. We do this by setting $q'=pq_2$ on $Y'$ for some $p\in\mathcal{P}$, and for each $v\in V(Y')$ we set $q'=\sigma_v q_2$ on $\lkblu(v)$ for some extendible bijection $\sigma_v:\lkblu(q_2(v))\to\lkblu(pq_2(v))$. The fact that the $\sigma_v$ are extendible ensures that $q'$ is extendible. However, $q'$ must agree with $q$ on the unique dashed edge $\tilde{\e}\in E^\pm(Z)$ with $\tilde{v}:=\tilde{\e}(0)\in V(Y')$. This gives us the following two constraints:
\begin{enumerate}[(1)]
	\item $pq_2(\tilde{v})=q(\tilde{v})$
	\item $\sigma_{\tilde{v}}q_2(\tilde{\e})=q(\tilde{\e})$
\end{enumerate}
The idea is then to meet these constraints by using the fact that $q$ is extendible. Indeed, extendibility of $q$ implies that the pair $(q(\tilde{\e}),q_2(\tilde{\e}))$ is extendible, so there exists $p\in\mathcal{P}$ satisfying (1). And Lemma \ref{claims12}(2) ensures that there exists an extendible bijection $\sigma_{\tilde{v}}:\lkblu(q_2(\tilde{v}))\to\lkblu(pq_2(\tilde{v}))$ satisfying (2). Finally, the existence of extendible bijections $\sigma_v$ for the other vertices $v$ follows from Lemma \ref{claims12}(1).
\end{proof}

\subsection{Final step in the proof of Theorem \ref{t:unimodularexample}}

To complete the proof of the theorem, we must prove that there does not exist a finite graph $\hat{X}$ fitting into the diagram (\ref{nohatX2}). Let $\hat{p}_i:\hat{X}\to X_i$ denote the coverings of $X_1$ and $X_2$. Just as in Theorem \ref{thm:cocompactexample},
 each connected
component $\hat{Y}$ of the solid subgraph of $\hat{X}$
 would fit into the following diagram of covers:
\begin{equation}\label{nohatY2}
\begin{tikzcd}[
ar symbol/.style = {draw=none,"#1" description,sloped},
isomorphic/.style = {ar symbol={\cong}},
equals/.style = {ar symbol={=}},
subset/.style = {ar symbol={\subset}}
]
&Y\ar{d}\\
&\hat{Y}\ar{dl}[swap]{\hat{p}_1}\ar{dr}{\hat{p}_2}\\
Y_1&&Y	
\end{tikzcd}
\end{equation}

From this it is immediate that $\hat{Y}\cong Y$, so $\hat{X}$ is a finite graph of copies of $Y$ joined together by dashed edges. We call the copies of $Y$ in $\hat{X}$ (i.e.~the connected components of the solid subgraph) \emph{pieces}. Each piece is equipped with a covering $p=\hat{p}_1\circ\hat{p}_2^{-1}:Y\to Y_1$ induced from the above diagram. Recall that $Y$ contains two triangles, called the left and right triangles (highlighted in bold below). We saw in (\ref{tri-a})
that $p$ will map either the left or the right triangle three times 
around the edge loop based at the vertex $a$ in $Y_1$; we call the piece a \emph{left piece} or a \emph{right piece} correspondingly.

For any oriented dashed edge $\e$ in $\hat{X}$, the existence of the maps $\hat{p}_1$ and $\hat{p}_2$ on the adjacent pieces tells us that the pair $(\hat{p}_1(\e),\hat{p}_2(\e))$ is extendible. The way we will arrive at a contradiction is to use lifts of the dashed edges highlighted green and red below to compute two different ratios between the number of left and right pieces in $\hat{X}$
(cf.~Example \ref{exmp:nouniformlattice}). Let $A_g$ [resp.~$A_r$] denote the set of lifts of the green [resp.~red] edges to $\hat{X}$ that map down to one of the orange edges in $X_1$.

\begin{figure}[H]
	\centering
\scalebox{0.6}{
\begin{tikzpicture}[auto,node distance=2cm,
thick,every node/.style={circle,draw,font=\small,minimum size=0.8cm},
every loop/.style={min distance=2cm},
hull/.style={draw=none}
]
\tikzstyle{label}=[draw=none,font=\Huge]

\begin{scope}[shift={(-7,0)},scale=1.5]
\node (1) {a};
\node (2) at (2,2) {b};
\node (3) at (2,-2) {c};	
\node[hull] (Y1) at (0,2.5) {};
\node[label] (l) at (1,3.6) {$X_1$};

\path
(1) edge [loop left] (1)
(1) edge (2)
(1) edge (3)
(2) edge (3)
(2) edge [bend left=20] (3)
(2) edge [bend right=20] (3)

(1) edge [orange,dashed,line width=3pt,bend left=30,
	postaction={decoration={markings,mark=at position 0.6 with {\arrow[orange,line width=1mm]{triangle 60}}},decorate}] (2)
(1) edge [orange,dashed,line width=3pt, bend left=50,
	postaction={decoration={markings,mark=at position 0.6 with {\arrow[orange,line width=1mm]{triangle 60}}},decorate}] (2)
(2) edge [blue,dashed, bend left=40] (3)
(2) edge [blue,dashed, bend left=50] (3)
(1) edge [orange,dashed,line width=3pt, bend right=30,
	postaction={decoration={markings,mark=at position 0.6 with {\arrow[orange,line width=1mm]{triangle 60}}},decorate}] (3)
(1) edge [orange,dashed,line width=3pt, bend right=50,
	postaction={decoration={markings,mark=at position 0.6 with {\arrow[orange,line width=1mm]{triangle 60}}},decorate}] (3);
	\end{scope}

\begin{scope}[shift={(7,0)},scale=2]

\node (1) {a};
\node (2) at (-1,1) {abc};
\node (3) at (1,1) {abc};
\node (4) at (-1,-1) {abc};
\node (5) at (1,-1) {abc};
\node (6) at (-2,2) {bc};
\node (7) at (2,2) {bc};
\node (8) at (-2,-2) {bc};
\node (9) at (2,-2) {bc};
\node[label] (X2) at (0,2.7) {$X_2$};

\node[label] (v1) at (-1.5,-1) {\huge $v_1$};
\node[label] (v2) at (0.5,-1) {\huge$v_2$};
\node[label] (w1) at (-1.5,1) {\huge$w_1$};
\node[label] (w2) at (0.5,1) {\huge$w_2$};

\path
(1) edge [ultra thick] (2)
(1) edge [ultra thick] (3)
(1) edge [ultra thick] (4)
(1) edge [ultra thick] (5)
(2) edge [ultra thick] (4)
(3) edge [ultra thick] (5)
(2) edge (6)
(3) edge (7)
(4) edge (8)
(5) edge (9)
(6) edge [bend left=10] (8)
(6) edge [bend right=10] (8)
(7) edge [bend left=10] (9)
(7) edge [bend right=10] (9)
(6) edge [bend left=30] (5)
(8) edge [bend right=30] (3)
(7) edge [bend right=30] (2)
(9) edge [bend left=30] (4)

(6) edge [blue,dashed,bend right=30] (8)
(6) edge [blue,dashed,bend right=40] (8)
(6) edge [blue,dashed,bend left=20] (1)
(6) edge [blue,dashed,bend left=30] (1)
(1) edge [blue,dashed,bend left=20] (8)
(1) edge [blue,dashed,bend left=30] (8)

(4) edge [green,dashed,line width=3pt,
	postaction={decoration={markings,mark=at position 0.6 with {\arrow[green,line width=1mm]{triangle 60}}},decorate},bend left=15] (2)
(4) edge [green,dashed,line width=3pt,
	postaction={decoration={markings,mark=at position 0.6 with {\arrow[green,line width=1mm]{triangle 60}}},decorate},bend left=35] (2)
(4) edge [green,dashed,line width=3pt,
	postaction={decoration={markings,mark=at position 0.6 with {\arrow[green,line width=1mm]{triangle 60}}},decorate},bend left=55] (2)
(2) edge [blue,dashed,bend left=20] (7)
(9) edge [blue,dashed,bend left=20] (4)
(7) edge [blue,dashed,bend left=20] (3)
(7) edge [blue,dashed,bend right=20] (3)
(5) edge [blue,dashed,bend left=20] (9)
(5) edge [blue,dashed,bend right=20] (9)
(5) edge [red,dashed,line width=3pt,
	postaction={decoration={markings,mark=at position 0.6 with {\arrow[red,line width=1mm]{triangle 60}}},decorate},bend right=25] (3)
(5) edge [red,dashed,line width=3pt,
	postaction={decoration={markings,mark=at position 0.6 with {\arrow[red,line width=1mm]{triangle 60}}},decorate},bend right=45] (3)

(9) edge [blue,dashed,bend right=40] (7);
\end{scope}

\end{tikzpicture}
}
\caption{\small Counting left and right pieces.}
\end{figure}
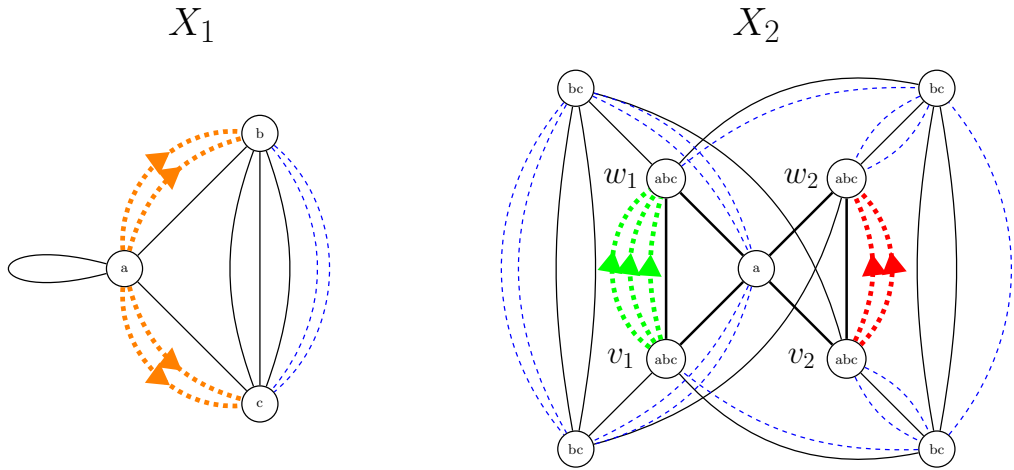

Each lift of $v_1$ to a left piece in $\hat{X}$ will map down to $a$ in $X_1$, so it has all three of its outgoing green edges in $A_g$. Each lift of $w_1$ to a right piece in $\hat{X}$ will map down to either $b$ or $c$ in $X_1$, and exactly two of its incoming green edges will be in $A_g$ since its other incident dashed edge must map to an edge in $X_1$ connecting $b$ to $c$.
 Each edge in $A_g$ has its initial vertex (the $v_1$ end)
 in a left piece and its terminus (the $w_1$ end)
  in a right piece, so we deduce the following equation.
$$3(\text{\# left pieces})=|A_g|=2(\text{\# right pieces})$$

Meanwhile, each lift of $v_2$ to a right piece in $\hat{X}$ will map down to $a$ in $X_1$, so both its outgoing red edges will be in $A_r$. And each lift of $w_2$ to a left piece will map down to either $b$ or $c$ in $X_1$, and both its incoming red edges will be in $A_r$ since its other incident dashed edges must map to edges connecting $b$ to $c$ in $X_1$. Each edge in $A_r$ has its 
origin (the $v_2$ end) in a right piece and its terminus (the $w_2$ end) in a left piece, so we deduce the following equation.
$$2(\text{\# left pieces})=|A_r|=2(\text{\# right pieces})$$

These two equations are inconsistent, so we get our desired contradiction, proving that $\hat{X}$ does not exist.
\qed
\bigskip

\begin{remk}
	We believe that our example for Theorem \ref{t:unimodularexample} is probably minimal among examples where $X_2$ decomposes into solid and dashed edges, with the solid edges forming a graph $Y$, and where $X$ is the regular cover of $X_2$ 
	corresponding to the normal subgroup $\<\!\<\pi_1Y\>\!\>$ of $\pi_1X_2$.
A clue as to why it might be minimal is the fact that the covering $Y\to Y_1$ that we use is the smallest example of a non-regular covering between non-vertex-transitive graphs. We do not give a proof that both these properties on $Y\to Y_1$ are necessary, but we 
can share an insight as to why choosing $Y=Y_1$ will not work. Indeed, if $Y=Y_1$ then we could take the cone
	$$Y\times[0,1]/[(y_1,0)\sim(y_2,0)],$$
	and glue it to $X_1$ by identifying $Y\subset X_1$ with $Y\times\{1\}$ to produce a simplicial complex $Z_1$;
	similarly we could glue the cone to each copy of $Y$ that appears in $X$, to produce a simply connected simplicial complex $Z$.
	The covering $X\to X_1$ would then extend to a covering $Z\to Z_1$, and this would be a regular covering since $Z$ is simply connected. We could then deduce that $X\to X_1$ was a regular covering, and find a finite graph $\hat{X}$ fitting into diagram (\ref{nohatX2}) by applying Theorem \ref{t:qL}.
\end{remk}

 \bigskip
\section{Space of covers}\label{s:measures}

We conclude with a technical result that extends our previous comments on the interplay between the regularity
of coverings, cocompactness and unimodularity. 
Bass and Kulkarni \cite{BL} prove that if $X$ is a locally finite tree on which ${\rm{Aut}}(X)$ acts
cocompactly, then ${\rm{Aut}}(X)$ is unimodular (in the sense that it supports a Haar measure that is
both left and right invariant) if and only if it contains a uniform lattice. We
shall not appeal to this fact, but we do need the more elementary fact that a
locally compact group that admits a lattice is unimodular; see \cite{rag}[Chapter I, Remark 1.9].

 In the following definition, as elsewhere in this article, we work in the category of metric graphs in which
 all edges have length $1$. We allow only morphisms that map vertices to vertices and restrict to local isometries
 on edges. In particular, this is true of the covering maps that we allow.

\begin{defn}
	Let $X$ be an infinite connected graph that covers a finite graph $Y$. 
	We define $\mathscr{C}(X,Y)$ to be the space of covering maps $X\to Y$;
	this is a closed subspace of the product space $E^\pm(Y)^{E^\pm(X)}$, endowed with the Tychonoff topology.
	\end{defn}
	
	If a group $G$ acts on $X$ then it acts continuously on the compact Hausdorff space 
	$\mathscr{C}(X,Y)$. Thus it makes sense to talk about a $G$-invariant probability measure on $\mathscr{C}(X,Y)$.

\begin{thm}\label{thm:probmeasure2}
Let $X$ be a locally finite infinite quasitree such that $\rm{Aut}(X)$ contains a uniform lattice, and suppose that
$X$ covers a finite graph $Y$. Then there exists an $\rm{Aut}(X)$-invariant probability measure on $\mathscr{C}(X,Y)$ if and only if there exists a finite cover $\hat{Y}\to Y$ and a regular covering map $X\to\hat{Y}$.
\end{thm}

\begin{remk}
Theorem \ref{t:unimodularexample} gives an example of a quasitree $X$ 
that covers finite graphs $X_1$ and $X_2$, with the covering $X\to X_2$ being regular, such that there is no covering $X\to\hat{X}$ of a finite graph $\hat{X}$ that covers $X_1$ and $X_2$.
There cannot exist a regular covering $X\to\hat{Y}$ of a finite graph $\hat{Y}$ that covers $X_1$, as then Theorem \ref{t:qL} would provide the aforementioned prohibited graph $\hat{X}$. Hence Theorem \ref{thm:probmeasure2} implies that $\mathscr{C}(X,X_1)$ admits no $\rm{Aut}(X)$-invariant probability measure.
\end{remk}

\begin{remk}
	Theorem \ref{thm:probmeasure2} is false if we remove the assumption that $\Aut(X)$ contains a uniform lattice. For instance if $\Aut(X)$ is compact then the right-invariant Haar measure on $\Aut(X)$ pushes forward to give an $\Aut(X)$-invariant probability measure on $\mathscr{C}(X,Y)$ via any orbit map $\Aut(X)\to\mathscr{C}(X,Y)$. But there cannot exist a regular covering of a finite graph $X\to\hat{Y}$, because its Galois group would be an infinite lattice in $\Aut(X)$, contradicting compactness of $\Aut(X)$. An example of such a covering $X\to Y$ is depicted in Figure \ref{fig:compactaut} - $\Aut(X)$ is compact because all automorphisms of $X$ must fix the unique loop.
\end{remk}

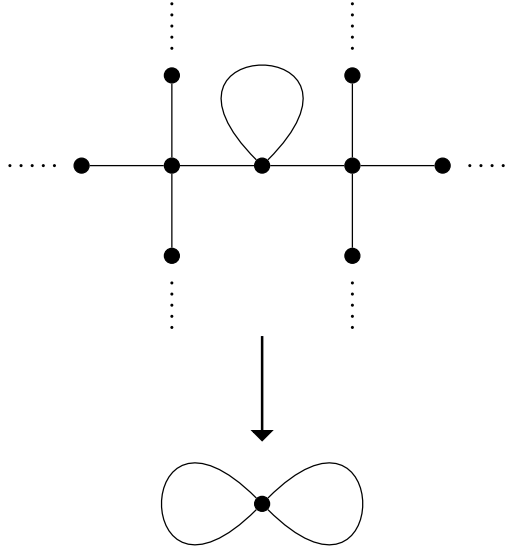
\begin{figure}[H]
	\centering
	\scalebox{0.6}{
		\begin{tikzpicture}[auto,node distance=2cm,
			thick,every node/.style={circle,fill,draw,font=\small},
			every loop/.style={looseness=40},
			]
			\tikzstyle{hull}=[draw=none, fill=none]
			\tikzstyle{dots}=[line width=2pt, line cap=round, dash pattern=on 0pt off 7pt, shorten <=12]
			
			\begin{scope}[scale=2]
				\node (1) {};
				\node (2) at (1,0){};
				\node (3) at (1,1){};
				\node (4) at (2,0){};
				\node (5) at (1,-1){};
				\node[hull] (6) at (1,2){};
				\node[hull] (7) at (3,0){};
				\node[hull] (8) at (1,-2){};
				\node (2a) at (-1,0){};
				\node (3a) at (-1,1){};
				\node (4a) at (-2,0){};
				\node (5a) at (-1,-1){};
				\node[hull] (6a) at (-1,2){};
				\node[hull] (7a) at (-3,0){};
				\node[hull] (8a) at (-1,-2){};
				
				\node[hull] (Y) at (0,-1.8) {};

				\path
				(1) edge [in=135,out=45,loop] (1)
				(1) edge (2)
				(2) edge (3)
				(2) edge (4)
				(2) edge (5)
				(3) edge [dots] (6)
				(4) edge [dots] (7)
				(5) edge [dots] (8)
				
				(1) edge (2a)
				(2a) edge (3a)
				(2a) edge (4a)
				(2a) edge (5a)
				(3a) edge [dots] (6a)
				(4a) edge [dots] (7a)
				(5a) edge [dots] (8a);
				
			\end{scope}
			
			\begin{scope}[shift={(0,-7.5)},scale=2]
				
				\node (1){};
				
				\node[hull] (Y1) at (0,0.6) {};
				
				\path
				(1) edge [in=45,out=-45,loop] (1)
				(1) edge [in=135,out=-135,loop] (1);
				
			\end{scope}

			\draw[draw=black,fill=blue,-triangle 90, ultra thick] (Y) -- (Y1);
			
		\end{tikzpicture}
	}
	\caption{\small A covering $X\to Y$ with $\Aut(X)$ compact.}\label{fig:compactaut}
\end{figure}

\noindent{\em{Proof of Theorem \ref{thm:probmeasure2}.}}
First suppose that there is a covering $p:X\to Y$ that factors through a regular covering $X\to\hat{Y}$ with $\hat{Y}$ finite, and let $\Gamma <\Aut(X)$ be the Galois group of $X\to\hat{Y}$. The existence of $\G$ shows that
$\Aut(X)$ is unimodular, so has a Haar measure $\mu$ which is both left and right invariant. As $\Gamma <\Aut(X)$ is a uniform lattice, the Haar measure on $\Aut(X)$ induces a finite right invariant measure $\bar{\mu}$ on the right coset space $(\Aut(X):\Gamma)$, which we may assume is a probability measure by rescaling. We then have a continuous right $\Aut(X)$-invariant map
\begin{align*}
\Phi:(\Aut(X):\Gamma)&\to\mathscr{C}(X,Y)\\
\Gamma\phi&\mapsto p\circ\phi.
\end{align*}	
This is well-defined because the action of $\Gamma$ on $X$ preserves fibres of
the covering map $X\to\hat{Y}$ (and hence also of $p:X\to Y$). Thus we can push forward the probability measure $\bar{\mu}$ on $(\Aut(X):\Gamma)$ to an $\Aut(X)$-invariant probability measure $\lambda$ on $\mathscr{C}(X,Y)$.
\bigskip
	
Conversely, suppose that we have an $\Aut(X)$-invariant probability measure $\lambda$ on $\mathscr{C}(X,Y)$. We apply Theorem \ref{maketree} to get an $\Aut(X)$-invariant quasi-isometry $f:X\to T$ to a tree $T$. $\Aut(X)$ contains a uniform lattice by assumption, and since $X$ is a quasitree we may pass to a free uniform finite-index sublattice $\Gamma <\Aut(X)$, and $\Gamma$ also acts freely and
cocompactly on $T$. By passing to a further finite-index sublattice if necessary, we may assume that the quotient $T/\Gamma$ has no 1-cycles.

We now define polyhedral pairs and face pairs, which we will eventually glue together to build the intermediate cover $\hat{Y}$; this construction is closely analogous to those in \cite{graphfins} and in the appendix of \cite{TwogenLeighton}.
For $v\in V(T)$, let  $\rm{star}(v)\subset T$ be the union of $v$ and the interiors of the edges incident to it, and let $P(v):=f^{-1}(\rm{star}(v))$, which we call the \emph{polyhedron at $v$}. (The term ``polyhedron" is suggestive but this set might be disconnected and need not be a subgraph of $X$; it will be bounded however.) By boundedness, we know that there are only finitely many maps $\eta:P(v)\to Y$ that extend to a cover $X\to Y$, and we denote this set of maps by $\Map(v)$.  For $e\in E(T)$, let $\mathring{e}$ be the interior of $e$, and define $F(e):=f^{-1}(\mathring{e})$ to be the \emph{face at $e$}, and let $\Map(e)$ be the (finite) set of maps $\eta: F(e)\to Y$ that extend to covers $X\to Y$.

We use square brackets to denote $\Gamma$-orbits of edges and vertices in $T$. Pick $v\in V(T)$, let $P[v]$ be a space isometric to $P(v)$ via an isometry $\alpha_v:P[v]\to P(v)$, and define $\Map[v]:=\Map(v)\circ\alpha_v$. For any $u\in V(T)$ with $[u]=[v]$ there is a unique $g\in\Gamma$ with $g(v)=u$; this means that $gP(v)=P(u)$, and we get a bijection
\begin{align*}
\Map(u)&\to\Map(v)\\
\eta&\mapsto \eta\circ g,
\end{align*}
because $\eta\in\Map(u)$ extending to $\bar{\eta}:X\to Y$ implies that $\eta\circ g$ extends to $\bar{\eta}\circ g$. Thus we can define $\alpha_u:=g\circ\alpha_v$ and  get   $\Map[v]=\Map(u)\circ\alpha_u$. Similarly, we have spaces $F[e]$ equipped with sets of maps $\Map[e]$, such that for any $e\in E(T)$ there is an isometry $\alpha_e:F[e]\to F(e)$, and  $\Map[e]=\Map(e)\circ\alpha_e$, and for $g\in\Gamma$ we have $\alpha_{g(e)}=g\circ\alpha_e$. For $v\in V(T)$ and $\phi\in\Map[v]$, we say that $\bfP=(P[v],\phi)$ is a \emph{polyhedral pair}, while for $e\in E(T)$ and $\varphi\in\Map[e]$ we say that $\bfF=(F[e],\varphi)$ is a \emph{face pair}. We write $\mathcal{P}$ and $\mathcal{F}$ for the collections of polyhedral pairs and face pairs respectively.

If $e$ is an edge incident at a vertex $v$ in $T$, then we have an inclusion $F(e)\xhookrightarrow{} P(v)$. We can then
 define an inclusion $\iota_{e,v}:F[e]\to P[v]$ via the following commutative diagram:
\begin{equation}
\begin{tikzcd}[
ar symbol/.style = {draw=none,"#1" description,sloped},
isomorphic/.style = {ar symbol={\cong}},
equals/.style = {ar symbol={=}},
subset/.style = {ar symbol={\subset}}
]
F[e]\ar{d}{\alpha_e}\ar[hook]{r}{\iota_{e,v}}&P[v]\ar{d}{\alpha_v}\\
F(e)\ar[r,hook]&P(v)
\end{tikzcd}
\end{equation}
Moreover, for $g\in\Gamma$ we have that $\iota_{g(e),g(v)}=\iota_{e,v}$. Since no $\Gamma$-translate of $e$ is incident at $v$ (by freeness of the action and the fact that $T/\Gamma$ has no 1-cycles), we conclude that $\iota_{e,v}$ only depends on $[e]$ and $[v]$; we write it as $\iota_{[e],[v]}$.

As $\Gamma$ acts freely on $T$, we can choose a $\Gamma$-invariant orientation on the edges of $T$, and for each $e\in E(T)$ we will refer to its two endpoints as the vertices on the \emph{left} and \emph{right} of $e$. If $v$ is on the left (resp. right) of $e$, and $\bfP=(P[v],\phi)$ and $\bfF=(F[e],\varphi)$ are such that $\phi\circ\iota_{[e],[v]}=\varphi$, then we say that $\bfP$ is the \emph{left} (\emph{right}) of $\bfF$, and denote this by $\bfP\in\overleftarrow{\bfF}$ ($\bfP\in\overrightarrow{\bfF}$). If $\bfP_1=(P[v_1],\phi_1)\in\overleftarrow{\bfF}$ and $\bfP_2=(P[v_2],\phi_2)\in\overrightarrow{\bfF}$, then we can glue $\bfP_1$ and $\bfP_2$ together along $\bfF$ and the maps $\iota_{[e],[v_1]}$ and $\iota_{[e],[v_2]}$, and we can define a map
$$\phi_1\cup\phi_2:P[v_1]\cup_{F[e]}P[v_2]\to Y.$$

To construct the intermediate cover $X\to\hat{Y}\to Y$, we will take $\omega(\bfP)$ copies of each polyhedral pair $\bfP$ and $\omega(\bfF)$ copies of each face pair $\bfF$, and each face pair will be glued to one polyhedral pair on the left and one polyhedral pair on the right. We will discuss the covering maps $X\to\hat{Y}$ and $\hat{Y}\to Y$ later, but for the gluing to even be possible the numbers $\omega(\bfP)$ and $\omega(\bfF)$ must satisfy the following Gluing Equations for each face pair $\bfF$:
\begin{equation}\label{GluingEquations}
\sum_{\bfP\in\overleftarrow{\bfF}}\omega(\bfP)=\omega(\bfF)=\sum_{\bfP\in\overrightarrow{\bfF}}\omega(\bfP)
\end{equation}
We note that by \cite[claim on p19]{Agoltheorem}, a non-zero solution $\omega:\mathcal{P}\cup\mathcal{F}\to\mathbb{R}_{\geq0}$ to the Gluing Equations implies the existence of a non-zero integer solution $\omega:\mathcal{P}\cup\mathcal{F}\to\mathbb{Z}_{\geq0}$. We now show how the $\Aut(X)$-invariant probability measure $\lambda$ on $\mathscr{C}(X,Y)$ can be used to solve the Gluing Equations.\\
\begin{claim}
	The weights
	\begin{align*}
	\omega(P[v],\phi)=\lambda(\{\theta\mid\theta|_{P(v)}=\phi\circ\alpha_v^{-1}\})\\
	\omega(F[e],\varphi)=\lambda(\{\theta\mid\theta|_{F(v)}=\varphi\circ\alpha_e^{-1}\})
	\end{align*}
	solve the Gluing Equations (\ref{GluingEquations}).
\end{claim}

	Firstly, note that the above formulae for $\omega$ are independent of the choice of vertex and edge in the orbits $[v]$ and $[e]$ because $\lambda$ is $\Gamma$-invariant. If $v$ is on the left of $e$ and $\bfF=(F[e],\varphi)$ is a face pair, then $(P[v],\phi)\in\overleftarrow{\bfF}$ if and only if $\phi\circ\alpha_v^{-1}:P(v)\to Y$ is an extension of the map $\varphi\circ\alpha_e^{-1}:F(e)\to Y$; and so 
	$$\{\theta\mid\theta|_{F(v)}=\varphi\circ\alpha_e^{-1}\}=\bigsqcup_{(P[v],\phi)\in\overleftarrow{\bfF}}\{\theta\mid\theta|_{P(v)}=\phi\circ\alpha_v^{-1}\}.$$
	Applying $\lambda$ to both sides yields
	$$\omega(\bfF)=\sum_{\bfP\in\overleftarrow{\bfF}}\omega(\bfP),$$
	as required. The argument works similarly if $v$ is on the right of $e$.
	Finally, the map $\omega:\mathcal{P}\cup\mathcal{F}\to\mathbb{Z}_{\geq0}$ is non-zero because
	$$1=\lambda(\mathscr{C}(X,Y))=\sum_{\phi\in\Map[v]}\lambda(\{\theta\mid\theta|_{P(v)}=\phi\circ\alpha_v^{-1}\})=\sum_{\phi\in\Map[v]}\omega(P[v],\phi)$$
	for $v\in V(T)$.\\

Now assume that $\omega:\mathcal{P}\cup\mathcal{F}\to\mathbb{Z}_{\geq0}$ is a non-zero integer solution to the Gluing Equations. We can then take $\omega(\bfP)$ copies of each polyhedral pair $\bfP$ and $\omega(\bfF)$ copies of each face pair $\bfF$, and glue them all together as described earlier to form a space $\hat{Y}$. This gluing can be described by a finite graph $Z$: take a vertex for each copy of a polyhedral pair and an edge for each copy of a face pair, then connect edges to vertices according to which face pairs get glued to which polyhedral pairs.
We can then define a covering $Z\to T/\Gamma$, where a vertex associated to $(P[v],\phi)$ maps to $[v]$ and an edge associated to $(F[e],\varphi)$ maps to $[e]$. The covering $T\to T/\Gamma$ then factors as $T\to Z\to T/\Gamma$ for some covering $T\to Z$, which on vertices and edges we denote by $v\mapsto\bar{v}$ and $e\mapsto\bar{e}$. It will be convenient to always refer to vertices and edges of $Z$ as images $\bar{v}$ and $\bar{e}$ of vertices and edges of $T$.
We can then write $\bfP_{\bar{v}}=(P[v],\phi_{\bar{v}})$ for the polyhedral pair associated to a vertex $\bar{v}\in V(Z)$ and $\bfF_{\bar{e}}=(P[e],\varphi_{\bar{e}})$ for the face pair associated to an edge $\bar{e}\in E(Z)$. If $\bar{v}$ is on the left (right) of $\bar{e}$, then $\bfP_{\bar{v}}$ will be glued to $\bfF_{\bar{e}}$ on the left (right) to form $\hat{Y}$ - in particular $\bfP_{\bar{v}}\in\overleftarrow{\bfF_{\bar{e}}}$ (resp. $\bfP_{\bar{v}}\in\overrightarrow{\bfF_{\bar{e}}}$). Therefore, we can write $\hat{Y}$ as
$$\hat{Y}=\faktor{\bigsqcup_{\bar{v}\in V(Z)}P[v]\times\{\bar{v}\}\sqcup\bigsqcup_{\bar{e}\in E(Z)}F[e]\times\{\bar{e}\}}{(\iota_{[e],[v]}(a),\bar{v})\sim (a,\bar{e})\text{ for $\bar{e}$ incident at $\bar{v}$.}}$$
Note that the images of the maps $\iota_{[e],[v]}$ in a polyhedron $P[v]$ are disjoint, so each point in $\hat{Y}$ either lies in a single polyhedron $P[v]\times\{\bar{v}\}$, or it lies in a face $F[e]\times\{\bar{e}\}$ as well as in the two polyhedra $P[v]\times\{\bar{v}\}$ with $\bar{e}$ incident at $\bar{v}$, but not in any other polyhedra. Since $T/\Gamma$ has no 1-cycles, each polyhedron $P[v]\times\{\bar{v}\}$ is embedded in $\hat{Y}$.

We define the map $\chi:X\to\hat{Y}$ by 
\begin{align*}
x\in P(v)&\mapsto (\alpha_v^{-1}(x),\bar{v})\\
x\in F(e)&\mapsto (\alpha_e^{-1}(x),\bar{e}),
\end{align*}
 for $v\in V(T)$ and $e\in E(T)$. Note that every point $x\in X$ lies in some polyhedron $P(v)$ because $f(x)\subset\rm{star}(v)$ for some $v\in V(T)$. If $e$ is incident at $v$ then the two maps above agree on $x\in F(e)\subset P(v)$ because $(\alpha_v^{-1}(x),\bar{v})=(\iota_{[e],[v]}\alpha_e^{-1}(x),\bar{v})\sim(\alpha_e^{-1}(x),\bar{e})$. Moreover, if $\chi(x)\in F[e]\times\{\bar{e}\}$, then $x$ is contained in a face in $X$, and the polyhedra on the left and right of this face will map down through $\chi$ to the two polyhedra $P[v]\times\{\bar{v}\}$ with $\bar{e}$ incident at $\bar{v}$. 
 
 We define the map $\psi:\hat{Y}\to Y$ by
 \begin{align*}
(a,\bar{v})\mapsto\phi_{\bar{v}}(a)\\
(a,\bar{e})\mapsto\varphi_{\bar{e}}(a),
 \end{align*}
 for $\bar{v}\in V(Z)$ and $\bar{e}\in E(Z)$. This is well defined, because if $\bar{e}$ is incident at $\bar{v}$ then $(a,\bar{e})\sim(\iota_{[e],[v]}(a),\bar{v})\mapsto \phi_{\bar{v}}\iota_{[e],[v]}(a)=\varphi_{\bar{e}}(a)$ since $\bfP_{\bar{v}}$ is on the left or right of $\bfF_{\bar{e}}$. It remains to show that these maps are coverings, and that $X\to\hat{Y}$ is a regular covering. Note it is not obvious from the above description that $\hat{Y}$ is even a graph, but this follows from the existence of the covering maps.\\
 
 \begin{claim}
 	$\chi:X\to\hat{Y}$ is a regular covering map.
 \end{claim}

We have a factoring of covers $T\to Z\to T/\Gamma$, so $Z=T/\hat{\Gamma}$ for $\hat{\Gamma}<\Gamma$ of finite index. If $g\in\hat{\Gamma}$ and $v\in V(T)$, then $\bar{v}=\overline{g(v)}$ and we get the following commutative diagram.
\begin{equation}
\begin{tikzcd}[
ar symbol/.style = {draw=none,"#1" description,sloped},
isomorphic/.style = {ar symbol={\cong}},
equals/.style = {ar symbol={=}},
subset/.style = {ar symbol={\subset}}
]
P(v)\ar[bend right=40]{ddr}[swap]{\chi}\ar{rr}{g}&&P(g(v))\ar[bend left=40]{ddl}{\chi}\\
&P[v]\ar{ul}[swap]{\alpha_v}\ar{d}\ar{ur}{\alpha_{g(v)}}\\
&P[v]\times\{\bar{v}\}
\end{tikzcd}
\end{equation}
Hence $\chi\circ g=\chi$.

 On the other hand, if $x_1\in P(v_1)$ and $x_2\in X$ satisfy $\chi(x_1)=\chi(x_2)$, then $x_2\in P(v_2)$ for some $v_2$ with $\bar{v}_1=\bar{v}_2$, and so there exists $g\in\hat{\Gamma}$ with $g(v_1)=v_2$. This also tells us that $g(x_1),x_2\in P(v_2)$ have the same image under $\chi$, and since the restriction of $\chi$ to $P(v_2)$ is injective we deduce that $g(x_1)=x_2$. $\hat{\Gamma}$ acts properly on $X$, so we conclude that $\chi$ is a covering map with deck transformation group $\hat{\Gamma}$.\\
\begin{claim}
	$\psi:\hat{Y}\to Y$ is a covering map.
\end{claim}

	As $\hat{Y}$ is compact (by the previous claim), it is enough to show that each point in $\hat{Y}$ has an open neighbourhood that maps homeomorphically onto an open subspace of $Y$. Given $y\in \hat{Y}$, take a lift $x\in X$ and consider $f(x)\in T$. We know that $f(x)\in\rm{star}(v)$ for some vertex $v$, so $x\in P(v)$. We also know that $y=\chi(x)=(\alpha_v^{-1}(x),\bar{v})$, and that we have a commutative diagram
	
	\begin{equation}
	\begin{tikzcd}[
	ar symbol/.style = {draw=none,"#1" description,sloped},
	isomorphic/.style = {ar symbol={\cong}},
	equals/.style = {ar symbol={=}},
	subset/.style = {ar symbol={\subset}}
	]
	P(v)\ar[bend right=40]{ddr}[swap]{\eta}\ar{rr}{\chi}[swap]{\sim}&&P[v]\times\{\bar{v}\}\ar[bend left=40]{ddl}{\psi}\\
	&P[v]\ar{ul}{\sim}[swap]{\alpha_v}\ar{d}{\phi_{\bar{v}}}\ar{ur}[swap]{\sim}\\
	&Y,
	\end{tikzcd}
	\end{equation}
	where $\sim$ indicates homeomorphisms, and $\eta$ is some map in $\Map(v)$. The map $\eta:P(v)\to Y$ extends to a cover $X\to Y$, so there is a neighbourhood of $x$ in $P(v)$ that maps homeomorphically to an open subspace of $Y$. Since $\chi:X\to\hat{Y}$ is a cover, we know that this neighbourhood of $x$ maps homeomorphically to a neighbourhood of $y$, which in turn is mapped homeomorphically by $\psi$ to the same open subspace of $Y$.\\
 
This completes the proof of the theorem. \qed


\vspace{1cm}
\begin{thebibliography}{} 



\bibitem{agol} I. Agol, 
{\em The virtual Haken conjecture} (with appendix by I. Agol,
D. Groves and J. Manning),   Documenta Math. {\bf 18}, 1045--1087, 2013.


\bibitem{Bass}
H. Bass, \textit{Covering theory for graphs of groups}, J. Pure Appl. Algebra 89, 3--47, 1993.
 


\bibitem{BK}  
H.~Bass and R.~Kulkarni,
{\em Uniform tree lattices},
J. Amer. Math. Soc. 3, 843--902, 1990.

\bibitem{BL}  
H.~Bass and A.~Lubotzky, 
{\em Tree lattices.} 
Progress in Mathematics, 176. Birkhäuser Boston, Inc., Boston, MA, 2001.

\bibitem{BehrstockNeumann}
Jason A. Behrstock and Walter D. Neumann, \emph{Quasi-isometric classification of non-geometric 3-manifold groups}, J. Reine Angew. Math., 669:101-120, 2012.

\bibitem{bestvina} M.~Bestvina,
{\em Geometric group theory and 3-manifolds hand in hand: the fulfillment of Thurston's vision},
Bull. Amer. Math. Soc. 51, 53--70, 2014.

\bibitem{BH}
M.R. Bridson and A. Haefliger, Metric Spaces of Non-Positive Curvature, Grundlehren der Mathematischen Wissenschaften,  Vol. 319, Springer-Verlag, Heidelberg-Berlin, 1999. 
 
 

\bibitem{BWilt} M.R.~Bridson and H.~Wilton,
{\em The triviality problem for profinite completions},
 Invent. Math. 202, 839--874, 2015. 

\bibitem{B4-annals} M.R.~Bridson, B. McReynolds, A.W.~Reid, R. Spitler,
{\em Absolute profinite rigidity and hyperbolic geometry},
Ann. of Math. (2) 192 (2020), no. 3, 679–719.
 
\bibitem{BM}
Marc Burger and Shahar Mozes, 
{\em Finitely presented simple groups and products of trees}, C. R. Acad. Sci. Paris 
S\'{e}r. I Math. 324, no. 7, 747--752, 1997

\bibitem{caprace} Pierre-Emmanuel Caprace,
{\em Finite and infinite quotients of discrete and indiscrete groups}, London Math. Soc. Lecture Note Ser., 455, Cambridge Univ. Press, Cambridge, 2019.
 
\bibitem{chatt-niblo}
I.~Chatterji and G.~Niblo,  
{\em From wall spaces to CAT$(0)$ cube complexes},
Internat. J. Algebra Comput. 15, 875--885, 2005.
 
\bibitem{gromov} M. Gromov, {\em Hyperbolic groups}, in ``Essays in Group Theory”, pp. 75–263, Math. Sci. Res. Inst. Publ. vol. 8, Springer, New York, 1987.

	
\bibitem{panel}
Mark F. Hagen and Nicholas W.M. Touikan, \textit{Panel collapse and its applications}, Groups Geom. Dyn. 13 (2019), no. 4, 1285–1334.

\bibitem{HP}
Fr\'ed\'eric Haglund and Fr\'ed\'eric Paulin,
{\em Simplicit\'{e} de groupes d'automorphismes d'espaces \`{a} courbure n{\'{e}}gative},
in The Epstein birthday schrift, 181--248, Geom. Topol. Monogr. 1, Coventry, 1998.


\bibitem{HW}
Fr\'ed\'eric Haglund and Daniel T. Wise, {\em Special cube complexes},
Geom. Funct. Anal. 17,  1551--1620, 2008.

\bibitem{Huang}
Jingyin Huang, \emph{Commensurability of groups quasi-isometric to RAAGs}, Invent. Math. 213, no. 3, 1179–1247, 2018.

\bibitem{leighton} 
Frank T. Leighton, {\em Finite common coverings of graphs}, J. Combin. Theory Ser. B 33, 231--238, 1982.
 
 \bibitem{MSW}
Lee Mosher, Michah Sageev, and Kevin Whyte, \textit{Quasi-actions on trees I.
	Bounded valence}, Annals of Mathematics, 158, 115--164, 2003.

\bibitem{neumann} Walter D. Neumann, 
{\em On Leighton's graph covering theorem}, 
Groups Geom. Dyn. 4, 863--872, 2010.


\bibitem{nica} 
Bogdan Nica,
{\em  Cubulating spaces with walls}, Algebr. Geom. Topol. 4, 297--309, 2004. 

\bibitem{rag}
M. S. Raghunathan, Discrete subgroups of Lie groups, Ergebnisse der Math., vol. 68,
Springer-Verlag, 1972. 

\bibitem{roller}
Martin Roller,
{\em Poc-sets, median algebras and group actions. An extended study of Dunwoody's construction and Sageev's theorem},
arXiv:1607.07747, 2016.

\bibitem{sageev} 
Michah Sageev, 
{\em Ends of group pairs and non-positively curved cube complexes},
 Proc. London Math. Soc. 71, 585--617, 1995.
 
        
\bibitem{sageev-pcmi} 
Michah Sageev, 
{\em CAT$(0)$ cube complexes and groups},
in ``Geometric group theory", 7--54, IAS/Park City Math. Ser., 21, Amer. Math. Soc., Providence, RI, 2014.


\bibitem{Agoltheorem}
Sam Shepherd, \textit{Agol's theorem on hyperbolic cubulations}, Rocky Mountain J. Math., to appear.

\bibitem{TwogenLeighton}
Sam Shepherd,
(with appendix by Giles Gardam and Daniel J. Woodhouse),
 \textit{Two generalisations of Leighton's Theorem}, arXiv:1908.008302019, 2019.
 
 
\bibitem{samThesis} Sam Shepherd, \emph{Finite covers of graphs and cube complexes}, PhD Thesis Oxford University, in preparation

\bibitem{sam-woodh} Sam Shepherd and Daniel J. Woodhouse, \emph{Quasi-isometric rigidity for graphs of virtually free groups with two-ended edge groups}, arXiv:2007.10034, 2020.

\bibitem{StarkWoodhouse} Emily Stark and Daniel J. Woodhouse, \emph{Action rigidity for free products of hyperbolic manifold groups}, arXiv:1910.09609, 2019.

\bibitem{wise} 
Daniel T. Wise, {\em Non-positively curved squared complexes: Aperiodic tilings and non- residually finite groups}, PhD Thesis Princeton University, ProQuest LLC, Ann Arbor, MI., 1996.


\bibitem{riches} Daniel T. Wise, {\em From Riches to RAAGs: 3-manifolds, Right-angled Artin Groups, and Cubical Geometry}, CBMS Regional Conference Series in Mathematics 117, Amer. Math. Soc., Providence, RI, 2012.


\bibitem{graphfins}
Daniel J. Woodhouse, \textit{Revisiting Leighton's Theorem with the Haar measure}, Math. Proc, Cambridge Philos. Soc., 1-9. doi:10.1017/S0305004119000550.

\end{thebibliography}
\end{document}